\documentclass[11pt]{article}

\usepackage{amsmath}
\usepackage{amsfonts}
\usepackage{amsthm}
\newtheorem{theorem}{Theorem}[section]
\newtheorem{lemma}[theorem]{Lemma}
\newtheorem{corollary}[theorem]{Corollary}
\newtheorem{proposition}[theorem]{Proposition}
\newtheorem{remark}[theorem]{Remark}

\usepackage[T1]{fontenc}
\usepackage{lmodern}
\makeatletter\usepackage{microtype}\g@addto@macro\@verbatim{\microtypesetup{act‌​ivate=false}}\makeatother

\usepackage{array}

\newcommand{\mR}{\mathbb{R}}
\newcommand{\mE}{\mathbb{E}}
\newcommand{\mC}{\mathbb{C}}
\newcommand{\mZ}{\mathbb{Z}}
\newcommand{\e}{{\mathrm{e}}} 
\newcommand{\Om}{{\Omega}}
\newcommand{\tT}{\widetilde{T}}
\newcommand{\cD}{\mathcal{D}}
\newcommand{\cM}{\mathcal{M}}
\newcommand{\cE}{\mathcal{E}}
\newcommand{\cF}{\mathcal{F}}

\newcommand{\cS}{\mathcal{S}}
\newcommand{\cH}{\mathcal{H}}
\newcommand{\cC}{\mathcal{C}}
\newcommand{\ux}{\underline{x}}
\newcommand{\uy}{\underline{y}}
\newcommand{\upx}{\partial_{\underline{x}}}

\usepackage[plainpages=false,hidelinks,bookmarks]{hyperref}

\newcommand*{\defeq}{\mathrel{\vcenter{\baselineskip0.5ex \lineskiplimit0pt
			\hbox{\scriptsize.}\hbox{\scriptsize.}}}=}

\numberwithin{equation}{section}

\begin{document}

\title{Generalized Fourier transforms arising from the enveloping algebras of $\mathfrak{sl}(2)$ and $\mathfrak{osp}(1|2)$}

	\author{H.\ De Bie\footnote{E-mail: {\tt Hendrik.DeBie@UGent.be}} \and R. Oste\footnote{E-mail: \texttt{Roy.Oste@UGent.be}} \and J. Van der Jeugt\footnote{E-mail: {\tt Joris.VanderJeugt@UGent.be}} }

	\vspace{10mm}
	\date{\small{H. De Bie:  Department of Mathematical Analysis\\ Faculty of Engineering and Architecture -- Ghent University}\\
		\small{Galglaan 2, 9000 Gent,
			Belgium}\\ \vspace{5mm}
		\small{R. Oste and J. Van der Jeugt: Department of Applied Mathematics, Computer Science and Statistics\\ Faculty of Sciences -- Ghent University\\
			Krijgslaan 281, 9000 Gent, Belgium}
	}

\maketitle

\begin{abstract}
	The Howe dual pair $(\mathfrak{sl}(2),O(m))$ allows the characterization of the classical Fourier transform (FT) on the space of rapidly decreasing functions as the exponential of a well-chosen element of $\mathfrak{sl}(2)$ such that the Helmholtz relations are satisfied.
	
	In this paper we first investigate what happens when instead we consider exponentials of elements of the universal enveloping algebra of $\mathfrak{sl}(2)$.  This leads to a complete class of generalized Fourier transforms, that all satisfy properties similar to the classical FT.  There is moreover a finite subset of transforms which very closely resemble the FT. We obtain operator exponential expressions for all these transforms by making extensive use of the theory of integer-valued polynomials. We also find a plane wave decomposition of their integral kernel and establish uncertainty principles. In important special cases we even obtain closed formulas for the integral kernels.
	
	In the second part of the paper, the same problem is considered for the dual pair $(\mathfrak{osp}(1|2),Spin(m))$, in the context of the Dirac operator on $\mR^m$. This connects our results with the Clifford-Fourier transform studied in previous work.
\end{abstract}

\section{Introduction}
\label{intro}

The classical Fourier transform (FT) over $\mR^m$ is given by
\[
(\cF  f)  (y) = \frac{1}{(2\pi)^{m/2}}  \int_{\mathbb{R}^m} \e^{i \langle x,y \rangle}  f(x)  \, \mathrm{ d}x
\]
and is of crucial importance in many aspects of harmonic analysis and signal processing.

Recently, various extensions of the FT have been investigated by rewriting the transform as an operator exponential. Indeed, introducing the operators
\begin{align*}
\Delta_x \defeq \sum_{i=1}^{m}\partial_{x_{i}}^{2}, \qquad
\lvert x\rvert^2 \defeq  \sum_{i=1}^{m}x_{i}^{2}, \qquad
\mathbb{E}_x \defeq \sum_{i=1}^{m} x_{i} \partial_{x_{i}}
\end{align*}
with $\Delta_x$ the Laplace operator and $\mE_x$ the Euler operator, one has
\begin{equation}
\label{opFT}
\cF = \e^{-i\frac{ \pi  }{4}m}\e^{i\frac{\pi}{4}(-\Delta_x + \lvert x\rvert^2)}
\end{equation}
which relates the transform with the representation theory of the Lie algebra $\mathfrak{sl}(2)$ as the operators $E =\lvert x\rvert^2/2$, $F =-\Delta_x/2$ and $H =\mE_x + m/2$ satisfy
\begin{equation}
\label{sl2rel}
\big[H,E\big] = 2E,\quad \big[H,F\big] = -2F,\quad \big[E,F\big] = H,
\end{equation}
see e.g.~\cite{Fol1989,Howe} for a detailed mathematical treatment.

By now well-established fields of research based on this observation include the fractional FT \cite{OZA} and the family of linear canonical transforms (LCTs) \cite{W}. Both have extensive practical applications in the design of optical systems and signal processing.

The observation in formula (\ref{opFT}) has also lead to many other theoretically oriented generalizations of the FT, by considering alternative realizations of $\mathfrak{sl}(2)$ in terms of differential or difference operators (or combinations thereof). Once such a realization is obtained, it is possible to define a generalized FT based on formula (\ref{opFT}). It is then a challenging question to find a concrete integral transform expression for this operator. 
Research on this topic has a long history: it has been investigated in both the continuous and the discrete case, and also for more complicated algebras than $\mathfrak{sl}(2)$, such as the superalgebra $\mathfrak{osp}(1|2)$. Important examples in the continuous case include the Dunkl transform \cite{deJ, Said} and the radially deformed Fourier transform \cite{KM1,KM2,Orsted2, DBRFT} as well as its Clifford deformation \cite{H12, DBOSS}. In more complicated geometries, the super Fourier transform \cite{DBS9} and a $q$-deformed version \cite{CSigma} have been investigated. In the discrete case we mention the fractional  Fourier-Kravchuk transform \cite{AW} and various further deformations \cite{J, J2}. For a more detailed review of the followed strategy and the results in this line of investigation we refer the reader to \cite{KM3,DBR}.

The present paper has a related yet different aim. We do not intend to change the realization of  $\mathfrak{sl}(2)$ in formula (\ref{sl2rel}), but instead wish to study precisely how unique the operator realization of the FT in (\ref{opFT}) is and to what extent the interplay of $\mathfrak{sl}(2)$ and its Howe dual $O(m)$ fixes the FT (see \cite{HT}). We are specifically interested in determining whether any other operators portray similar behaviour.

A crucial property of the classical FT is its interaction with differential operators. In particular, we have for $j=1,\dotsc,m\colon$
\begin{align}
	\begin{split}\label{Fxi}
		\mathcal{F} \circ \partial_{x_{j}}&= - i \,y_{j}  \circ \mathcal{F}\\
		\mathcal{F} \circ x_{j}&= - i \,  \partial_{y_{j}}\circ \mathcal{F},
	\end{split}
\end{align}
which uniquely determines the FT and its integral kernel up to a normalization constant. 
We can relax \eqref{Fxi} 
to a more general interaction with differential operators featuring now the elements of the $\mathfrak{sl}(2)$ realization that appear in the operator exponential formulation \eqref{opFT} of the FT (which also takes into account the $O(m)$ symmetry). 
Hereto, let 
\[
(Tf)(y) = \frac{1}{(2\pi)^{m/2}} \int_{\mR^m} K(x,y)f(x) \, \mathrm{ d}x
\] be an integral transform on the space of rapidly decreasing functions $\cS(\mR^m) \subset L^2(\mR^m)$; we say that $T$ satisfies the Helmholtz relations if
\begin{align}
\begin{split}\label{Helmh}
T \circ  \Delta_{x}  & =-  \lvert y \rvert^2\circ T \\
T \circ   \lvert x\rvert^2 &  = -  \Delta_{y} \circ T\rlap{\,.}
\end{split}
\end{align}
These relations no longer have the classical FT as unique solution for the operator $T$. Only when imposing the extra condition that $T$ must be the exponential of an element of $\mathfrak{sl}(2)=\mathrm{span}\big \{\Delta_x,\lvert x \rvert^2,\big [\Delta_x,\lvert x \rvert^2 \big]\big \}$ do they yield as unique solution the FT (or its inverse) up to a normalization constant \cite{Howe, KM3}.  

The aim of our paper is to look for generalized Fourier transforms, satisfying properties similar to the classical FT. We do this by investigating what other solutions the Helmholtz relations \eqref{Helmh} have, where we make explicit use of the $\mathfrak{sl}(2)$ relations in (\ref{sl2rel}) and their $O(m)$ invariance. 
Indeed, we wish to analyse in detail what happens when considering $T$ an operator exponential from the universal enveloping algebra of $\mathfrak{sl}(2)$. 
If we combine this with a periodicity restriction on the eigenvalues, as is the case for the FT, we are led to an interesting finite set of new generalized Fourier transforms that behave in a way very similar to the classical FT. 
In \hyperref[theo4per]{Theorem~\ref*{theo4per}}, we establish for all transforms in this set their operator exponential formulation. Subsequently, in \hyperref[theoTabc]{Theorem~\ref*{theoTabc}}, we are even able to determine explicit integral kernels for some of these transforms, thus realizing them as integral transforms. A crucial role for reaching these results is played by the Casimir $\Omega$ of $\mathfrak{sl}(2)$ and the (very technical) study of integer-valued polynomials in $\Omega$.

In the more general context of Clifford-algebra valued functions, where the Dirac operator takes over the role of the Laplace operator (see \cite{DSS, GM}) and the Lie superalgebra $\mathfrak{osp}(1|2)$ that of $\mathfrak{sl}(2)$, quite a bit of attention has already been paid to transforms consisting of a specific operator exponential from the universal enveloping algebra of $\mathfrak{osp}(1|2)$. Most notably, the so-called Clifford-Fourier transform \cite{MR2190678, MR2283868, DBNS, DBXu} is defined as such an operator exponential. 

To provide a solid motivation for the study of this transform, we will also treat the case of Clifford analysis with the tools developed in the present paper. 
We again find an interesting class of generalized, now essentially non-scalar Fourier transforms that behave similarly to the classical FT, see \hyperref[theo4perC]{Theorem~\ref*{theo4perC}} for their operator exponential formulation and \hyperref[theoTab]{Theorem~\ref*{theoTab}} for explicit kernels realizing them as integral transforms.

The resulting generalized FTs that we obtain in both the harmonic and Clifford case exhibit properties and behaviour similar to the FT. To emphasise that, we will also show that they e.g.~still satisfy a version of the uncertainty principle (see e.g.~\cite{Fol} for a review). This is achieved in \hyperref[theoUP]{Theorem~\ref*{theoUP}} and \hyperref[theoUPC]{Theorem~\ref*{theoUPC}}.

The paper is organized as follows. 
We commence \hyperref[sec:2]{Section~\ref*{sec:2}} by laying out the specific properties we will use as a starting point to determine generalized FTs and introducing the relevant $\mathfrak{sl}(2)$ and $O(m)$ representation theory.  
Next, we work out the requirements for our operators which brings us to the concept of integer-valued polynomials. These polynomials allow us to give an explicit expression for the desired operators. 
For a subset of solutions which satisfy a periodicity restriction we give a complete classification and we also obtain explicit integral kernels for these transforms.
Finally, we establish a generalized uncertainty principle. 
In \hyperref[sec:3]{Section~\ref*{sec:3}} we lift our objective to the setting of Clifford analysis 
and apply the strategy we developed in the previous section to obtain analogous results. In particular, we again give a complete classification of the subset of solutions which satisfy a periodicity restriction and we obtain explicit integral kernels. 
Finally, in \hyperref[sec:4]{Section~\ref*{sec:4}} we present some conclusions regarding our results, while in the \hyperref[sec:5]{appendix} we give an overview of main definitions and results used in the main text, as well as some technical proofs that have been omitted from the text.

\section{Fourier transforms in harmonic analysis}
\label{sec:2}

We want to look for generalized Fourier transforms, satisfying similar properties to the classical FT. As interaction with differential operators we require the Helmholtz relations \eqref{Helmh} and we will further impose two additional important properties of the FT. In particular, we require the eigenfunction basis of the space of rapidly decreasing functions $\cS(\mR^m) \subset L^2(\mR^m)$ to be preserved, which will be explained in more detail shortly. 

This gives our first goal, which is to determine all operators
$
T \colon 
\mathcal{S}(\mathbb{R}^m)  \to \mathcal{S}(\mathbb{R}^m) 
$
that satisfy the following properties:
\begin{itemize}
	\item[(i)] the Helmholtz relations
	\begin{flalign*}
	&T \circ  \Delta_{x}   \ =-  \lvert y \rvert^2\circ T &&\\
	&T \circ   \lvert x\rvert^2\,  = -  \Delta_{y} \circ T&&
	\end{flalign*}
	\item[(ii)] $ T \, \phi_{j,k,\ell} = \mu_{j,k} \, \phi_{j,k,\ell} \qquad \text{with }\mu_{j,k}\in \mC$
	
	\item[(iii)] $T^4 = \operatorname{id}$
\end{itemize}
Here, the standard eigenfunction basis of $\cS(\mR^m) \subset L^2(\mR^m)$ is given by the Hermite functions
\begin{equation}\label{basis}
\phi_{j, k, \ell}\defeq 2^j j!\, L_{j}^{\frac{m}{2} + k-1}(|x|^2) H_k^{(\ell)} \, \e^{-|x|^{2}/2},
\end{equation}
where $j, k \in \mZ_{\geq0}$,  $L_j^{\alpha}$ is a generalized Laguerre polynomial and $\{\,H_k^{(\ell)} \mid \ell= 1, . . . , \dim(\cH_k) \,\}$ is a basis for $\cH_k$, the space of spherical harmonics of degree $k$, i.e.~homogeneous polynomial null-solutions of the Laplace operator of degree $k$. This basis $\{ \phi_{j, k, \ell}\}$ realizes the complete decomposition of $\cS(\mR^m)$ in irreducible subspaces under the natural action of the dual pair $(\mathfrak{sl}(2),O(m))$, see e.g.~\cite{Orsted2}. 
The action of the Fourier transform on the eigenfunction basis is given by 
\begin{equation}\label{eigF}
\cF \phi_{j,k,\ell} = \e^{i\frac{\pi}{2}(2j+k)} \, \phi_{j,k,\ell} = i^{2j+k} \, \phi_{j,k,\ell}.
\end{equation}
Another way to write the Hermite functions (see \cite{CPLMS}) is 
\begin{equation}\label{ephi}
\phi_{j,k,\ell} =
\left(-\frac{\Delta_x}{2} - \frac{\lvert x \rvert^2}{2} +\mE_x + {\frac{m}{2}}\right)^j
H_{k}^{(\ell)} \, \e^{-\lvert x\rvert^{2}/2}.
\end{equation}

Now, consider the following linear combinations of the operators $E,F,H$ satisfying \eqref{sl2rel} 
\begin{equation}\label{hef}
h= -\frac{\Delta_x}{2} + \frac{\lvert x \rvert^2}{2}, \quad
e=-\frac{\Delta_x}{4} - \frac{\lvert x \rvert^2}{4} +\frac12 \mE_x + {\frac{m}{4}},\quad
f =\frac{\Delta_x}{4} + \frac{\lvert x \rvert^2}{4} +\frac12 \mE_x + { \frac{m}{4}}.
\end{equation}
This triple generates another operator realization of the Lie algebra $\mathfrak{sl}(2)$ (see also \cite{Said,Orsted2} for a more general situation). 
Indeed, by means of \eqref{sl2rel} 
one easily verifies that they  
satisfy the commutation relations
\begin{equation}
\label{sl2re}
\lbrack h,e \rbrack = 2e, \qquad
\lbrack h,f \rbrack = -2f, \qquad
\lbrack e,f \rbrack = h.
\end{equation}
Using these operators, 
the Helmholtz property~{(i)} translates to the  
(anti-)commutation relations
\begin{equation}\label{Thef}
T \circ h  =h\circ T, \qquad T \circ e = -e\circ T, \qquad T \circ f = -f\circ T.
\end{equation}
Moreover, in terms of these operators we can write property \eqref{ephi} more compactly as
$
\phi_{j,k,\ell} =
\big(2e\big)^j
\phi_{0,k,\ell} 
$, 
and the operator exponential formulation of the classical Fourier transform as 
\begin{equation}
\label{opF}
\cF =  \e^{i \frac{\pi}{4}( -\Delta_x +|x|^{2}-m ) } =  \e^{i \frac{\pi}{2}  ( h-\frac{m}{2} )   }.
\end{equation}
The action of the operators \eqref{hef} on the basis \eqref{basis} is as follows
\[
h \, \phi_{j,k,\ell}=\big(2j+k+{\tfrac{m}{2}}\big)\, \phi_{j,k,\ell},\quad
e \, \phi_{j,k,\ell}=\frac{1}{2}\, \phi_{j+1,k,\ell},\quad
f \,\phi_{j,k,\ell}=-j\, (2j-2+m+2k)\, \phi_{j-1,k,\ell}.
\]
Note that $h$ is a diagonal operator while $e$ works as a raising operator on the $j$-index and $f$ as a lowering operator.
\begin{remark}
	For every $k\in \mZ_{\geq0}$ and $\ell\in \{ 1, \dotsc , \dim(\cH_k)\}$, the set $\{\, \phi_{j,k,\ell} \mid j \in \mZ_{\geq0}  \, \}$ forms a basis for the positive discrete series representation of $\mathfrak{sl}(2)$ with lowest weight $k+m/2$. This is an irreducible unitary representation of the real form $\mathfrak{su}_{1,1}$ of $\mathfrak{sl}(2)$ 
	(determined by the $\ast$-conditions: $h^\ast = h$, $e^\ast = -f$, $f^\ast = -e$).
	
	For every $j,k\in \mZ_{\geq0}$ the set $\{\, \phi_{j,k,\ell} \mid \ell= 1, . . . , \dim(\cH_k)  \, \}$ forms a basis for an irreducible representation of $O(m)$ for $m>2$.
\end{remark}
Before moving on to the explicit computation of solutions for the operator $T$, we first glance at some important consequences that hold for every
operator satisfying the properties~{(i)}--{(iii)}.
\begin{lemma}\label{eigenv}
	For $T$ an operator that satisfies the properties~{(i)}--{(iii)}, 
	there are only four possible values for the eigenvalues $
	\mu_{j,k}$ of $T$, namely 
	$\mu_{j,k} \in \{1,i,-1,-i\}.
	$
	Moreover, the spectrum of eigenvalues is completely determined by the eigenvalues $\mu_{0,k}$ for $k\in \mZ_{\geq0}$; the other eigenvalues for $j>0$ follow from the relation
	\begin{equation}\label{mujk}
	\mu_{j,k}  = (-1)^{j}\mu_{0,k}.
	\end{equation}
\end{lemma}
\begin{proof}
	Property~{(iii)} necessitates that the eigenvalues $\mu_{j,k}$ of $T$ satisfy $(\mu_{j,k})^4=1$ and thus are integer powers of $i=\e^{i\frac{\pi}{2}}$.
	Using property~{(ii)} and the relations \eqref{ephi} and \eqref{Thef}, we find
	\[
	\mu_{j+1,k} \, \phi_{j+1,k,\ell} = T\, \phi_{j+1,k,\ell} = T \circ (2e) \, \phi_{j,k,\ell} =  -2e \circ T  \, \phi_{j,k,\ell} = -2e\,  \mu_{j,k}  \, \phi_{j,k,\ell} = -\mu_{j,k} \, \phi_{j+1,k,\ell},
	\]
	for all valid $j,k,\ell$. 
	The relation \eqref{mujk} now follows from subsequent application of $ \mu_{j+1,k} = -\mu_{j,k}$.
\end{proof}

\begin{proposition}\label{intr}
	Let $
	T \colon 
	\mathcal{S}(\mathbb{R}^m)  \to \mathcal{S}(\mathbb{R}^m) 
	$ be an operator that satisfies the properties~{(i)} and {(ii)}. 
	Then, on the basis $\{\phi_{j,k,\ell}\}$, the operator $T$ 
	coincides with the integral transform 
	\[
	(T f)  (y) = \frac{1}{(2\pi)^{m/2}} \int_{\mathbb{R}^m}\! K_m(x,y) \, f(x) \, \mathrm{ d}x
	\]
	where
	\begin{equation}\label{kern}
	K_m(x,y) =  {2}^{\lambda}\, \Gamma ( \lambda )    \sum_{k=0}^{+\infty} (k+\lambda)\, \mu_{0,k} \, z^{-\lambda}J_{k+\lambda}( z) \,C^{\lambda}_{k}(w).
	\end{equation}
	Here, the following notations are used: $\lambda=(m-2)/2$, $z=|x |  |y|$, $w = \langle x,y \rangle/z$, and $J_{k+\lambda}$ are Bessel functions while $C_k^{\lambda}$ denote the so-called Gegenbauer or ultraspherical polynomials (see Appendix~\ref{secA1}). 
\end{proposition}
\begin{proof}
	For $ \xi, \eta \in \mathbb{S}^{m-1}$ and $H_\ell \in \cH_\ell$, the following reproducing kernel formula holds (see e.g.~\cite{MR1827871})
	\begin{equation}\label{reprod}
	\frac{1}{(2 \pi)^{m/2}} \int_{\mathbb{S}^{m-1}} {2^{\lambda}}\, {\Gamma ( \lambda)}\, (k +\lambda) \, C_k^{\lambda} ( \langle \xi, \eta \rangle) H_{\ell}(\xi) \, \mathrm{d}\xi = \delta_{k\ell} \ H_\ell(\eta).
	\end{equation}
	Using the explicit form \eqref{basis} of the eigenfunctions, together with formula \eqref{reprod} and the identity 
	\begin{equation}\label{lagbes}
	\int_{0}^{\infty} r^{2\lambda+1+k}  
	L_{j}^{k+\lambda}(r^{2}) \, e^{-r^{2}/2} \, (rs)^{-\lambda}\, J_{k+\lambda}(rs)\,\mathrm{d}r = (-1)^{j}\, s^{k}\,  L_{j}^{k+\lambda}(s^{2})\,  e^{-s^{2}/2}
	\end{equation}(see e.g.~\cite[exercise 21, p.\,380]{Sz}), it follows that
	\[
	\frac{1}{(2\pi)^{m/2}} \int_{\mathbb{R}^m}\! K_m(x,y)  \, \phi_{j,k,\ell}(x) \, \mathrm{ d}x = (-1)^j\, \mu_{0,k}  \, \phi_{j,k,\ell}(y).
	\]
	By relation \eqref{mujk} for the eigenvalues of $T$, this integral transform coincides with $T$ on the Hermite basis.
\end{proof}

\begin{lemma}\label{homogenity}
	The kernel $K_m$ satisfies
	\begin{align}
	\label{Hom_K}
	\begin{split}
	K_m(Ax, y) &= K_m(x, Ay), \qquad A \in O(m)\\
	K_m(c x, y) & = K_m(x, cy), \qquad c \in \mR.
	\end{split}
	\end{align}
\end{lemma}
\begin{proof}
	This follows from the explicit formula (\ref{kern}) for $K_m(x,y)$.
\end{proof}

\begin{proposition}\label{unitary}
	A continuous operator $T\colon 
	\mathcal{S}(\mathbb{R}^m)  \to \mathcal{S}(\mathbb{R}^m) $ satisfying {(i)}--{(iii)} has a unitary extension to $L^{2}(\mathbb{R}^m)$.
\end{proposition}
\begin{proof}
	This result follows from the fact that $\mathcal{S}(\mathbb{R}^m)$ is dense in $L^{2}(\mathbb{R}^m)$ and that all eigenvalues have unit norm.
\end{proof}

We now continue to determine new operators $T$ that satisfy properties~{(i)}--{(iii)}. We already know 
one solution, namely the classical Fourier transform.  
In order to find more solutions, we first introduce a new operator as follows. 
Assume we have an operator $T$ that satisfies {(i)}--{(iii)}, we then put
$\tT \defeq T \circ \cF^{-1}$.
As the classical Fourier transform is an automorphism on $\mathcal{S}(\mathbb{R}^m)$, the operator $\tT$ uniquely defines $T$ and 
we can retrieve the operator $T$ by the relation
\begin{equation}\label{tTF}
T = \tT \circ \cF.
\end{equation}
Our objective to obtain operators $T$  
is thus equivalent with finding suitable operators $\tT$.
Hereto we determine what the requirements are for such an operator $\tT$ in order to yield an operator $T$ that satisfies properties {(i)}--{(iii)}. 

Clearly, the basis $\{\phi_{j,k,\ell}\}$ must also form an eigenbasis of $\tT$.
Next, as $\cF$ already satisfies the set of equations \eqref{Thef}, we have that 
an operator $T$ 
decomposed as \eqref{tTF} will satisfy the set of equations \eqref{Thef} if the operator $\tT$ commutes with the operators $h,e,f$.
A consequence of imposing this commutative property on $\tT$  
combined with $\cF$ being of the form \eqref{opF}, is that $\tT$ also commutes with $\cF$.
Property~{(iii)} together with $\cF^4 =  \operatorname{id}$ then gives us $\tT^4 =  \operatorname{id}$. 
These requisites for the operator $\tT$ naturally lead to 
the following result:
\begin{proposition}\label{proptT}
	Every operator $\tT$ of the form
	\begin{equation}\label{tTeO}
	\tT = \exp\! \big(  i\tfrac{\pi}{2} F \big),
	\end{equation} 
	with $F$ an operator that 
	\begin{itemize}
		\item commutes with (the generators of) $\mathfrak{sl}(2)=\mathrm{span}\big \{\Delta_x,\lvert x \rvert^2,\big [\Delta_x,\lvert x \rvert^2 \big]\big \}$ 
		\item has integer eigenvalues on the functions $\{\phi_{j,k,\ell}\}$ (independent of $\ell$)
	\end{itemize}
	will yield an operator $T$ by \eqref{tTF} that satisfies the properties {(i)}--{(iii)}.
\end{proposition}
Furthermore, in the next section we will show that for each operator satisfying the properties~{(i)}--{(iii)}, we have an equivalent operator of the form \eqref{tTeO}.

Now, we want to establish 
an explicit expression for the operator $F$ in \eqref{tTeO}. Hereto we start by looking at the first condition listed in \hyperref[proptT]{Proposition~\ref*{proptT}}, which requires commuting operators. 
We denote by $ \mathcal{U} \big( \mathfrak{sl}(2) \big)$ the {universal enveloping algebra} of $\mathfrak{sl}(2)$.
The center of $ \mathcal{U} \big( \mathfrak{sl}(2) \big)$ is the subset consisting of the elements that commute with all elements of $\mathfrak{sl}(2)$, and hence also with all elements of $ \mathcal{U} \big( \mathfrak{sl}(2) \big)$. The center is finitely generated by the Casimir element (see \cite{Humphreys}): 
\begin{equation}\label{Om}
\Omega = 1+ h^2 +2ef+2fe,
\end{equation}
or in the framework of our operator realization: 
$
\Omega 
= \big(\mE_x + {\tfrac{m-2}{2}}\big)^2-\lvert x \rvert^2{\Delta_x}. 
$ 
Every polynomial function of the operator $\Omega$ will yield an operator $F$ that commutes with the generators of $\mathfrak{sl}(2)$. 
This notion can be further generalized to include infinite power series in $\Om$. Such operators live in the extension $ \overline{\mathcal{U}} \big( \mathfrak{sl}(2) \big)$ of the universal enveloping algebra that also allows infinite power series in the elements of $\mathfrak{sl}(2)$ \cite{Jacobson}.

The second condition in \hyperref[proptT]{Proposition~\ref*{proptT}} is also facilitated by operators of this form as the Casimir element 
is a diagonal operator on the representation space $\mathrm{span}\{\, \phi_{j,k,\ell} \mid j \in \mZ_{\geq0}  \, \}$. 
The eigenvalues of $\Om$ are 
given by 
\begin{equation}\label{evO}
\Omega \, \phi_{j,k,\ell}=\Bigl( k+\frac{m}{2}-1\Bigr)^2\, \phi_{j,k,\ell} =( k+\lambda)^2\, \phi_{j,k,\ell}   ,
\end{equation}
with $\lambda=m/2-1$. 
Note that for even dimensions $m$, the eigenvalues of $\Om$ are squares of integers, while for odd dimension, we have squares of half-integers (elements of $\mZ\!+\!\frac12$).
In order to find expressions for $F$ that have integer eigenvalues on the functions $\{\phi_{j,k,\ell}\}$, we need functions that take on integer values when evaluated at the eigenvalues of $\Om$.
To this end, we invoke the notion of integer-valued polynomials.
This is the subject of the following subsection.

\subsection{Integer-valued polynomials}\label{ivp}

An {integer-valued polynomial} on $\mZ$ is a polynomial whose value at every integer $n\in \mZ$ is again an integer. We denote the set of all such polynomials by 
$
\mathrm{Int}(\mZ) = \big\{\, f \in \mathbb{Q}[x] \mid f(\mZ) \subseteq \mZ \,\big\}.
$
It is an elementary result that the polynomials
\begin{equation}\label{xchn}
\binom{x}{n}=  \prod_{\ell=0}^{n-1} \frac{x-\ell}{n-\ell} = \frac{1}{n!} \prod_{\ell=0}^{n-1} (x-\ell)
\end{equation}
are integer-valued; moreover, they form a basis of the $\mZ$-module $\mathrm{Int}(\mZ) $ (see e.g.~\cite{IVP}).
The polynomial $\binom{x}{n}$ has degree $n$ and its roots are the integers $\{0,1,\dotsc,n-1\}$. 
The first few polynomials 
are given by
\[
\binom{x}{0} = 1,\qquad\binom{x}{1} = x, \qquad\
\binom{x}{2} = \frac12x^2- \frac12x, \qquad
\binom{x}{3} = \frac{1}{6}x^3 - \frac{1}{2}x^2 +  \frac{1}{3}x.
\]

Now, we are interested in functions that take on integer values when evaluated at the eigenvalues of the Casimir operator $\Om$. 
As there is a disparity pertaining to the form of the eigenvalues \eqref{evO}, depending on the dimension $m$, we first handle the case where the dimension is even.

\subsubsection{On squares of integers}

For even dimension $m$, the eigenvalues \eqref{evO} of $\Om$ are squares of integers, which are of course again integers, so every integer-valued polynomial with $\Om$ substituted for $x$ is a valid solution for $F$ in \eqref{tTeO}.
However, the condition to be integer-valued on squares of integers is less restrictive than the requirement of being integer-valued on all integers.
In order to specify the exact class of solutions, we introduce a special type of integer-valued polynomials. 
For $n=0$, put $E_0(x)  \equiv 1$, and for $n\in\mZ_{\geq1}$, put
\begin{equation}\label{Enx}
E_n(x) 
=  \prod_{\ell=0}^{n-1} \frac{x^2-\ell^2}{n^2-\ell^2}
= \frac{2}{(2n)!} \prod_{\ell=0}^{n-1} (x^2-\ell^2).
\end{equation}
The polynomial $E_n$ has degree $2n$ and its roots are the integers $\{0,\pm1,\pm2,\dotsc,\pm(n-1)\}$, while $
E_n(n) = 1.
$
The first few of these polynomials are given by
\[
E_0(x) = 1,\qquad
E_1(x) = x^2,  \qquad
E_2(x)  = \frac{1}{12}x^4 - \frac{1}{12}x^2, \qquad E_3(x)  = \frac{1}{360}x^6 - \frac{1}{72}x^4 +  \frac{1}{90}x^2.
\]
\begin{proposition}
	The polynomials $\{\, E_n(x) \mid n\in\mZ_{\geq0} \, \}$ are integer-valued on $\mZ$.
\end{proposition}
\begin{proof}
	The case $n=0$ is obvious. For $n\geq1$ one has
	\begin{equation*}
	E_n(x) = \frac{2}{(2n)!} \prod_{\ell=0}^{n-1} (x^2-\ell^2) = \frac{x}{n}  \frac{1}{(2n-1)!} \prod_{\ell=0}^{2n-2} (x-\ell+n-1)= \frac{x}{n} \binom{x+n-1}{2n-1}.
	\end{equation*}
	Using a property of binomial coefficients we get
	\[
	\frac{x}{n} \binom{x+n-1}{2n-1} = \frac{x+n-n}{n} \binom{x+n-1}{2n-1} =2 \binom{x+n}{2n} - \binom{x+n-1}{2n-1},
	\]
	which is clearly  
	integer-valued.
\end{proof}

The prominent feature of the polynomials defined in \eqref{Enx} is that they contain only even powers of $x$. 
Because of this, substituting the operator $\Om$ for $x^2$ in $E_n(x)$ yields an operator whose eigenvalues when acting on the eigenfunctions $\{\phi_{j,k,\ell}\}$ are all integers. 
We use the notation
\[E_n(\sqrt{\Om})= E_n(x)\Big\rvert_{x^2 = \Om}\] 
to denote this substitution. Moreover, we have the following important property. 
\begin{proposition}
	Every integer-valued polynomial $P(x)$ that contains only even powers of $x$ can be written as a $\mZ$-linear combination of the polynomials $E_n(x)$, i.e. 
	\begin{equation}\label{anEN}
	P(x) = \sum_{n=0}^{N} a_n E_n(x), \qquad a_n \in \mZ.
	\end{equation}
\end{proposition}
\begin{proof}
	Let  $P(x)$ be an integer-valued polynomial that contains only even powers of $x$. The polynomial $P(x)$ necessarily has even degree, say $2N$.
	Now, we know that the polynomial $E_n$ has degree $2n$ and its roots are the integers $\{0,\pm1,\pm2,\dotsc,\pm(n-1)\}$, while $
	E_n(n) = 1.
	$
	Hence, the coefficients $a_0,a_1,\dotsc,a_N\in \mZ$ in the sum \eqref{anEN}
	can be recursively determined from the (integer) values $P(0),P(1),\dotsc,P(N)$ such that when evaluated at $0\leq x \leq N$ the sum \eqref{anEN} equals $P(x)$.
	Moreover, as $P(-x)=P(x)$, \eqref{anEN} coincides with $P(x)$ on at least $2N+1$ points. 
	A polynomial of degree $2N$ is uniquely determined by its values at $2N+1$ points, which completes the proof.
\end{proof}
Note that every polynomial in $x$ that is integer-valued on squares can be turned into a polynomial that is integer-valued on $\mZ$ and that contains only even powers in $x$ (by substituting $x^2$ for $x$).
Hence, we have shown that the polynomials $E_n(x)$ suffice to construct every polynomial that takes on integer values when evaluated at the eigenvalues of the Casimir operator $\Om$.

One can further generalize the previous concept to construct any integer-valued function $F$ satisfying $F(x)=F(-x)$,
by specifying the coefficients $a_n \in \mZ$ in the (possibly infinite) series
\begin{equation*}
\sum_{n=0}^{\infty} a_n E_n(x).
\end{equation*}
Indeed, when $x=0$ this series equals the coefficient $a_0$, while its value at $x=n$ is fixed by the coefficients $a_0,a_1,\dotsc,a_n$ and thus can be 
specified by the choice of the value of $a_n$. 
Note that when evaluated at an integer $x\in \mZ$, only a finite number of terms in this series is nonzero, as $E_n(x) = 0$ for $n >\lvert x \rvert $. 

To conclude, we remark that there is an additional important aspect we have to consider for the goal we have in mind.  In the operator formulation \eqref{tTeO}, the function $F$ is used as an exponent of $\e^{i\frac{\pi}{2}}=i$.
For an integer $n$, the relation
$
i^n = i^{n \!\!\mod{ 4}}
$ holds. 
Hence, given two functions such that on each square number their function values are integers in the same 
congruence class modulo 4, they will yield the same eigenvalues and thus equivalent operators $\tT$. 
It thus suffices to consider the series 
\begin{equation}\label{anEn}
\sum_{n=0}^{\infty} a_n E_n(x), \qquad a_n \in \{0,1,2,3\},
\end{equation}
which gives all possible functions modulo 4. 

The values of the polynomials $E_n(x)$ 
modulo 4 can be computed by means of a computer algebra package. For the first few polynomials these are given in Table~1.
\begin{table}[!htbp]
	\[
	\begin{array}{c|*{16}c}
	x & 0&1 &2&3&4&5&6&7 & 8 &9&10&11&12&13&14&
	{15}\\ \hline 
	E_0 & 1 &1 &1 &1 &1 &1 &1 & 1 &1&1&1&1&1&1&1&1\\
	E_1 & 0 & 1 & 0 & 1 & 0 &1 & 0&1&0&1&0&1&0&1&0&1\\
	E_2 & 0&0&1&2&0&2&1&0&0&0&1&2&0&2&1&0\\
	E_3& 0&0&0&1&0&3&0&2&0&2&0&3&0&1&0&0\\
	E_4 & 0&0&0&0&1&2&2&2&0&2&2&2&1&0&0&0\\
	E_5 & 0&0&0&0&0&1&0&1&0&3&0&3&0&2&0&2
	\end{array}
	\]
	\caption
	{The values of the polynomials $E_0,E_1,\dotsc,E_5$ modulo 4 for $x=0,\dots,15$.}
\end{table}
Here, one can observe that $E_n(x)$ is periodic in $x$ with period dependent on $n$. The exact relation is 
stated in \hyperref[periodE]{Corollary~\ref*{periodE}} (Appendix).

\subsubsection{On squares of half-integers}

In the odd-dimensional case, the eigenvalues \eqref{evO} of $\Om$ are squares of half-integers. Therefore, we define the following polynomials. For $n\in\mZ_{\geq1}$, put
\begin{equation}\label{Dnx}
D_n(x)  =   \prod_{\ell=0}^{n-1} \frac{x^2-(\ell\!+\!\tfrac12)^2}{(n\!+\!\tfrac12)^2-(\ell\!+\!\tfrac12)^2}=  \frac{1}{(2n)!} \prod_{\ell=0}^{n-1} \big(x^2-(\ell\!+\!\tfrac12)^2\big).
\end{equation}
and for $n=0$ put $D_0(x)  \equiv 1$.
These polynomials are integer-valued on 
half-integers. Indeed, for $n\geq 1$ and $k\in \mZ$ one has
\[
D_n\bigl(k+\tfrac12\bigr) 
=  \frac{1}{(2n)!} \prod_{\ell=0}^{2n-1} \big(k+\tfrac12+n-1-\ell+\tfrac12\big) = \binom{k+n}{2n}.
\]
The first few polynomials are given by $D_0(x) = 1,$
\[
D_1(x) =  \frac{1}{2}x^2 -  \frac{1}{8},  \quad
D_2(x)  = \frac{1}{24}x^4 - \frac{5}{48}x^2 +\frac{3}{128},  \quad
D_3(x)  = \frac{1}{720}x^6 - \frac{7}{576}x^4 +  \frac{259}{11520}x^2- \frac{5}{1024}.\]
The polynomial $D_n$ has degree $2n$ and its roots are the half-integers $\big\{\pm\tfrac12,\pm\big(1\!+\!\tfrac12\big),\pm\big(2\!+\!\tfrac12\big),\dotsc,\pm\big(n-\!\tfrac12\big)\big\}$, while $
D_n(n\!+\!\tfrac12) = 1.
$

The values of the first few polynomials $D_n(x)$ 
modulo 4 are given in Table~2. Again, one clearly perceives the periodicity of $D_n(x)$, 
as stated in \hyperref[periodD]{Corollary~\ref*{periodD}} (Appendix).

\begin{table}[!htbp]
	\[
	\begin{array}{c|*{16}c}
	x\!-\!1/2 & 0&1 &2&3&4&5&6&7 & 8 &9&10&11&12&13&14&15\\ \hline 
	D_0 & 1 &1 &1 &1 &1 &1 &1 & 1 &1&1&1&1&1&1&1&1\\
	D_1 & 0 & 1 & 3 & 2 &2 &3 & 1&0&0&1&3&2&2&3&1&0\\
	D_2 &0& 0 & 1 & 1 & 3 & 3 &2 & 2&2&2&3&3&1&1&0&0\\
	D_3& 0&0&0&1&3&0&0&2&2&0&0&3&1&0&0&0\\
	D_4 &0&0&0&0&1&1&1&1&3&3&3&3&2&2&2&2\\
	D_5& 0&0&0&0&0&1&3&2&2&1&3&0&0&2&2&0
	\end{array}
	\]
	\caption{The values of the polynomials $D_0,D_1,\dotsc,D_5$ modulo 4 for $x=\tfrac12,\dots,15\!+\!\tfrac{1}{2}$.}
\end{table}
Similar to what we had for even dimension, the polynomials $D_n(x)$ 
contain only even powers of $x$, so we can again substitute the Casimir $\Om$ for $x^2$ in $D_n(x)$. 
In this way we arrive at the following general form 
\begin{equation}\label{anDn}
\sum_{n=0}^{\infty} a_n D_n(x), \qquad a_n \in \{0,1,2,3\}.
\end{equation}
Note that when evaluated at a half-integer $x\in \mZ\!+\!\tfrac12$, only a finite number of terms in this series is nonzero, as again $D_n(x) = 0$ for $n> \lvert x \rvert$.

\subsubsection{Conclusion}
\label{Concl}

In this way, we arrive at the following result for arbitrary dimension $m$.
\begin{theorem}\label{sols}
	The properties
	\begin{itemize}
		\item[(i)] the Helmholtz relations
		\begin{flalign*}
		&T \circ  \Delta_{x}   \ =-  \lvert y \rvert^2\circ T &&\\
		&T \circ   \lvert x\rvert^2\,  = -  \Delta_{y} \circ T&&
		\end{flalign*}
		\item[(ii)] $ T \, \phi_{j,k,\ell} = \mu_{j,k} \, \phi_{j,k,\ell} \qquad \text{with }\mu_{j,k}\in \mC$
		
		\item[(iii)]$T^4 = \operatorname{id}$
	\end{itemize}
	are satisfied by an operator $T$ of the form  
	\begin{equation}\label{TeO}
	T = \e^{i\frac{\pi}{2} F(\sqrt{\Om})
	}\e^{i \frac{\pi}{2}  ( h-\frac{m}{2} )   }  \in \overline{ \mathcal{U} } \big(  \mathfrak{sl}(2) \big),
	\end{equation}
	where $F(\sqrt{\Om})$ is an operator that consists of a function given by \eqref{anEn} (for even dimension) or \eqref{anDn} (for odd dimension)
	with the Casimir operator $\Omega$ substituted for $x^2$. 
	Conversely, every operator that satisfies properties~{(i)}--{(iii)} is equivalent with an operator of the form \eqref{TeO}.
\end{theorem}
\begin{proof}
	Only the last part remains to be proven. For this, it suffices to note that for an operator satisfying properties~{(i)}--{(iii)}, by \hyperref[eigenv]{Lemma~\ref*{eigenv}}, its spectrum of eigenvalues is completely determined by the eigenvalues $\mu_{0,k}$ for $k\in \mZ_{\geq0}$. Now, 
	the coefficients $a_n$ in \eqref{anEn} or \eqref{anDn} can be chosen to yield every possible set of valid eigenvalues $\mu_{0,k}$ for $k\in \mZ_{\geq0}$. 
	Indeed, recursively working upwards from $k=0$, the coefficient $a_{\lfloor k+\lambda \rfloor}$ fixes the value of the function $F$ evaluated at $k+\lambda$, with $\lambda=m/2-1$. 
	This in turn fixes the eigenvalue $\mu_{0,k}$ of $T$ 
	as we have
	\[
	T  \, \phi_{0,k,\ell} 
	=i^{ F(k+\lambda) }  i^{ k}\, \phi_{0,k,\ell}\rlap{,}
	\]
	where we used that the eigenvalues of $\Om$ are given by \eqref{evO}.
\end{proof}
Note that the preceding theorem naturally contains the classical Fourier transform. Indeed, for $F\equiv 0$ the operator \eqref{TeO} is precisely the operator exponential formulation \eqref{opF} of the Fourier transform. 
We now investigate how we can further narrow down this class of integral transforms $T$, preferably to a finite set of interesting transforms. 
Throughout this process we seek inspiration in other useful properties of the Fourier transform.

\subsection{Periodicity restriction}
\label{ssPr}

We now assume that $T$ is an operator of the form \eqref{TeO} as described in \hyperref[sols]{Theorem~\ref*{sols}}. 
The behavior of the eigenvalues $\mu_{j,k}$ of $T$ with regard to the index $j$ 
is given in \hyperref[eigenv]{Lemma~\ref*{eigenv}}.
By successive application of relation \eqref{mujk} we find that the eigenvalues are two-periodic in $j$:
\[
\mu_{j+2,k} = (-1)^2\, \mu_{j,k}=  \mu_{j,k}.
\]
As far as the $k$ index is concerned there are thus far no restrictions on the eigenvalues of the solutions \eqref{TeO} for $T$. 
From \eqref{eigF} one clearly sees that the eigenvalues of the Fourier transform are four-periodic in the index $k$. 
If we restrict the class of operators in \hyperref[sols]{Theorem~\ref*{sols}} to those operators whose eigenvalues are four-periodic in the index $k$,  we necessarily have but a finite number of valid operator solutions.
Indeed, from \hyperref[eigenv]{Lemma~\ref*{eigenv}} we know that every operator $T$ whose eigenvalues are four-periodic in $k$, i.e.~
\begin{equation*}
\mu_{j,k+4} =  \mu_{j,k},
\end{equation*}
will have its eigenvalue spectrum completely determined by the four values
$
\mu_{0,0},\mu_{0,1},\mu_{0,2},\mu_{0,3}
$ (or any other set of four eigenvalues with $k$ indices that are mutually incongruent modulo 4) . 
Furthermore, each one of these four eigenvalues can take on only four possible values, namely $\{ 1,i,-1,-i\}$, as by \hyperref[eigenv]{Lemma~\ref*{eigenv}} they must be integer powers of $i$. 
Hence, this leaves us with a finite number of valid operator solutions.

An additional advantage of imposing this periodicity restriction on $T$ is that this will allow us to obtain a closed formula for the kernel when $T$ is written as an integral transform. This will be discussed subsequently in Section \ref{Cftk}.

First, to find these solutions, we again look at the decomposition \eqref{tTF} of $T$. 
As we now require the eigenvalues of the operator $T$ to be four-periodic in $k$, and using that those of the Fourier transform already are four-periodic in $k$, this implies that the eigenvalues of the operator $\tT$ must also be four-periodic in $k$. 
In order for this to hold, we need a function of the form \eqref{anEn} or \eqref{anDn} that is four-periodic modulo 4 when evaluated at (the square root of) the eigenvalues of $\Om$.

Depending on the parity of the dimension $m$ one works in, the desired operators follow from the results in the following theorems. 
Here, we denote by $ \mathbf{1}_{A}$
the indicator function of the set $A$, defined as
\[  \mathbf{1}_{A} \colon A \to \{0,1\}  \colon
x \mapsto \mathbf{1}_{A}(x) = \begin{cases} 1 & \text{if }x \in A \\ 0 &  \text{if }x\notin A \end{cases}
\]
and also $4\mZ +b = \{\, 4a +b  \mid a \in \mZ \, \}$. For even dimension, the theorem is formulated as follows.
\begin{theorem} \label{moddb}
	For $x \in \mZ_{\geq0}$, one has
	\[
	E_1(x) \equiv 
	\mathbf{1}_{4\mZ+1}(x) + \mathbf{1}_{4\mZ+3}(x)   \pmod{ 4} 
	\]
	\[
	E_2(x) + 2\, E_3(x)  \equiv \mathbf{1}_{4\mZ+2}(x) \pmod{ 4}
	\]
	\[
	\sum_{n=1}^{\infty}\big( E_{2^n+1}(x) + \sum_{j=1}^{n-1} 2\, E_{2^n+1+2^j}(x) \big)  \equiv 
	\mathbf{1}_{4\mZ+3}(x)  \pmod{ 4} .
	\]
	If one denotes  
	\[
	\cE_{0101}(x) \defeq E_1(x), \qquad \cE_{0010}(x) \defeq E_2(x) + 2\, E_3(x),  \]\[
	\cE_{0001}(x) \defeq\sum_{n=1}^{\infty}\big( E_{2^n+1}(x) + \sum_{j=1}^{n-1} 2\, E_{2^n+1+2^j}(x) \big),
	\]
	then modulo 4 every integer-valued even function $F(x)$ on the integers that is four-periodic in $x$ can be written as
	\begin{equation}\label{FE}
	F(x) = a 
	+ b \,\cE_{0101} (x)  + c \,\cE_{0010} (x) +d\,\cE_{0001} (x),
	\end{equation}
	with $a,b,c,d \in \{0,1,2,3\}$.
\end{theorem}
\begin{proof}
	The proof of these results involves some long technical calculations and has therefore been omitted from the main text. It can be found in Appendix A.2. 
\end{proof}

To illustrate this property, the first few values of these functions modulo 4 are listed in Table~3.
\begin{table}[!htbp]
	\[
	\begin{array}{c|*{16}c}
	x & 0&1 &2&3&4&5&6&7 & 8 &9&10&11&12&13&14&15\\ \hline 
	1 & 1 &1 &1 &1 &1 &1 &1 & 1 &1&1&1&1&1&1&1&1\\
	\cE_{0101} & 0 & 1 & 0 & 1 & 0 &1 & 0&1&0&1&0&1&0&1&0&1\\
	\cE_{0010} & 0&0&1&0&0&0&1&0&0&0&1&0&0&0&1&0\\
	\cE_{0001} & 0&0&0&1&0&0&0&1&0&0&0&1&0&0&0&1 
	\end{array}
	\]
	\caption{The values of the functions $\cE_{0101} ,\cE_{0010},\cE_{0001}$ modulo 4 for $x=0,\dots,15$.}
\end{table}

In the same fashion, we have for odd dimension

\begin{theorem} \label{moddc}
	For $x \in \mZ_{\geq0}$, one has
	\[
	D_1\big(x\!+\!\tfrac12\big) + 2\,D_2\big(x+\tfrac12\big) \equiv \mathbf{1}_{4\mZ+1}( x )+
	\mathbf{1}_{4\mZ+2}(x) \pmod{ 4} 
	\]
	\[
	\sum_{n=1}^{\infty}\Big( D_{2^n} \big(x+\tfrac12\big) + \sum_{j=1}^{n-1} 2\, D_{2^n+2^j} \big(x+\tfrac12\big)\Big)    \equiv 
	\mathbf{1}_{4\mZ+2}(x) +\mathbf{1}_{4\mZ+3}(x) \pmod{ 4}
	\]
	\[
	2D_3 \big(x+\tfrac12\big)+ \sum_{n=0}^{\infty}\Big( D_{2^n+1} \big(x+\tfrac12\big) + \sum_{j=1}^{n-1} 2\, D_{2^n+1+2^j} \big(x+\tfrac12\big)\Big) \equiv
	\mathbf{1}_{4\mZ+2}(x) \pmod{ 4}.
	\]
	If one denotes 
	\[
	\cD_{0110}(x) \defeq D_1\big(x\!+\!\tfrac12\big) + 2\,D_2\big(x+\tfrac12\big) , \quad \cD_{0011}(x) \defeq  \sum_{n=1}^{\infty}\Big( D_{2^n} \big(x+\tfrac12\big) + \sum_{j=1}^{n-1} 2\, D_{2^n+2^j} \big(x+\tfrac12\big)\Big) , \]\[
	\cD_{0010}(x) \defeq2D_3 \big(x+\tfrac12\big)+ \sum_{n=0}^{\infty}\Big( D_{2^n+1} \big(x+\tfrac12\big) + \sum_{j=1}^{n-1} 2\, D_{2^n+1+2^j} \big(x+\tfrac12\big)\Big),
	\]
	then modulo 4 every integer-valued even function $F(x)$ on the half-integers that is four-periodic in $x$ can be written as
	\begin{equation}\label{FD}
	F(x) = a 
	+ b \,\cD_{0110} (x)  + c \,\cD_{0011} (x) +d\,\cD_{0010} (x),
	\end{equation}
	with $a,b,c,d \in \{0,1,2,3\}$.
\end{theorem}
\begin{proof}
	The proof of these results is similar to that of \hyperref[moddb]{Theorem~\ref*{moddb}}.
\end{proof}

The first few values of these functions modulo 4 are listed in Table~4.
\begin{table}[!htbp]
	\[
	\begin{array}{c|*{16}c}
	x\!-\!1/2 & 0&1 &2&3&4&5&6&7 & 8 &9&10&11&12&13&14&15\\ \hline 
	1 & 1 &1 &1 &1 &1 &1 &1 & 1 &1&1&1&1&1&1&1&1\\
	\cD_{0110} & 0 & 1&1& 0 &0 &1 & 1&0&0&1&1&0&0&1&1&0\\
	\cD_{0011}&0& 0 & 1 & 1 &0 & 0 &1 & 1&0&0&1&1&0&0&1&1\\
	\cD_{0010}& 0&0&1&0&0&0&1&0&0&0&1&0&0&0&1&0
	\end{array}
	\]
	\caption{The values of the functions $\cD_{0110} ,\cD_{0011},\cD_{0010}$ modulo 4 for $x=\tfrac12,\dots,15\!+\!\tfrac{1}{2}$.}
\end{table}

\begin{remark}
	The polynomials occurring in the preceding two theorems which are not an infinite series have as explicit form
	\[
	\cE_{0101}(x) = x^2, \qquad \cE_{0010}(x) = \frac{1}{180} x^6   + \frac{1}{18} x^4   -\frac{11}{180} x^2
	, \qquad
	\cD_{0110} (x) = \frac{1}{12} x^4   +\frac{7}{24} x^2  - \frac{5}{64}.
	\]
\end{remark}

Putting everything together, 

we can summarize our results in the following theorem.

\begin{theorem}\label{theo4per}
	Let $
	T \colon 
	\mathcal{S}(\mathbb{R}^m)  \to \mathcal{S}(\mathbb{R}^m) 
	$ be an operator that satisfies the following properties:
	\begin{itemize}
		\item[(i)] the Helmholtz relations
		\begin{flalign*}
		&T \circ  \Delta_{x}   \ =-  \lvert y \rvert^2\circ T &&\\
		&T \circ   \lvert x\rvert^2\,  = -  \Delta_{y} \circ T&&
		\end{flalign*}
		\item[(ii)] $ T \, \phi_{j,k,\ell} = \mu_{j,k} \, \phi_{j,k,\ell} \qquad \text{with }\mu_{j,k}\in \mC$
		
		\item[(iii)] $T^4 = \operatorname{id}$
		\item[(iv)] the eigenvalues of $T$ are 4-periodic in the index $k$: $ \mu_{j,k+4} =  \mu_{j,k}$.
	\end{itemize}
	Then $T$ can be written as
	\begin{equation*}
	T = \e^{i\frac{\pi}{2} F(\sqrt{\Om})
	}\e^{i \frac{\pi}{2}  ( h-\frac{m}{2} )   }
	\end{equation*}
	where $F$ consists of a function as specified in \eqref{FE} (\hyperref[moddb]{Theorem~\ref*{moddb}}) if $m$ even and \eqref{FD} (\hyperref[moddb]{Theorem~\ref*{moddc}}) if $m$ odd.
	
	Conversely, every operator $T$ of this form satisfies properties $(i)-(iv)$.
\end{theorem}
\begin{proof}
	The four eigenvalues $
	\mu_{0,0},\mu_{0,1},\mu_{0,2},\mu_{0,3}$ completely determine the eigenvalue spectrum of the operator $T$. By \hyperref[moddb]{Theorem~\ref*{moddb}} or \hyperref[moddc]{Theorem~\ref*{moddc}}, depending on the parity of the dimension, we can construct a function $F$ such that the eigenvalues of $\e^{i\frac{\pi}{2} F(\sqrt{\Om})
	}\e^{i \frac{\pi}{2}  ( h-\frac{m}{2} )   }$ and $T$ coincide. 
	The rest is a direct consequence of 
	\hyperref[sols]{Theorem~\ref*{sols}}.
\end{proof}

\subsection{Closed formulas for the kernel}
\label{Cftk}

In \hyperref[intr]{Proposition~\ref*{intr}}
we already found a formulation as an integral transform for all of the operator exponentials we obtained. 
The kernel of these integral transforms, given in \eqref{kern}, consists of an infinite series of Bessel functions and Gegenbauer polynomials.
Now, a natural question is whether it is possible to reduce these infinite series to a closed formula.
For a select case of transforms that also satisfy \hyperref[theo4per]{Theorem~\ref*{theo4per}} the answer to this is indeed positive. Our approach consists of first determining a formula for the kernel in the lowest dimension, followed by using a recursive relation to move up in dimension. 
Hereto, we first prove the following lemma.

\begin{lemma}\label{lemmak}
	Let $T$ be an operator of the form \eqref{TeO} as specified in \hyperref[sols]{Theorem~\ref*{sols}}. 
	When $T$ is written as an integral transform, its kernel satisfies the following recursive relation
	\[
	K_{m+2} = - i\, z^{-1} \partial_w K_m
	\]
	for $m\geq 2$. Here $K_{m+2}$ denotes the kernel in dimension $m+2$.
\end{lemma}
\begin{proof}
	As $T$ is of the form \eqref{TeO}, its eigenvalues are given by
	\[
	T  \, \phi_{j,k,\ell}  = \e^{i\frac{\pi}{2} F(\sqrt{\Om})
	}\e^{i \frac{\pi}{2}  ( h-\frac{m}{2} )   } \, \phi_{j,k,\ell} = \e^{i\frac{\pi}{2} F(k+\lambda) }  \e^{i\frac{\pi}{2}(2j+ k) }\, \phi_{j,k,\ell},
	\]
	where we used \eqref{evO} for the eigenvalues of $\Om$, \eqref{eigF} for those of $\cF$ and $\lambda=(m-2)/2$.
	
	From \hyperref[intr]{Proposition~\ref*{intr}} we know that
	$T$ can be written as an integral transform 
	\[
	(T f)   (y) = \frac{1}{(2\pi)^{m/2}} \int_{\mathbb{R}^m}\! K_m(x,y) \, f(x) \, \mathrm{ d}x
	\]
	with kernel 
	\begin{equation*}
	K_m(x,y) =  {2}^{\lambda}\, \Gamma ( \lambda )    \sum_{k=0}^{+\infty} (k+\lambda)\, i^{ F(k+\lambda) }  i^{k } \, z^{-\lambda}J_{k+\lambda}( z) \,C^{\lambda}_{k}(w).
	\end{equation*}
	Using Appendix \ref{secA1}, property \eqref{Geg0} of the Gegenbauer polynomials, we find 
	\begin{align*}
	- i\, z^{-1} \partial_w K_m(x,y)  & =  {2}^{\lambda}\, \Gamma ( \lambda )    \sum_{k=0}^{+\infty} (k+\lambda)\, i^{ F(k+\lambda) }  (-i) \, i^{k } \, z^{-\lambda-1}J_{k+\lambda}( z) \, \partial_w  C^{\lambda}_{k}(w) \\
	& =  {2}^{\lambda}\, \Gamma ( \lambda )    \sum_{k=1}^{+\infty} (k+\lambda)\, i^{ F(k+\lambda) }  i^{k-1 } \, z^{-(\lambda+1)}J_{k+\lambda}( z) \,2 \lambda \, C^{\lambda+1}_{k-1}(w)\\
	& =  {2}^{\lambda+1}\, \Gamma ( \lambda +1)    \sum_{k=0}^{+\infty} (k+1+\lambda)\, i^{ F(k+1+\lambda) }  i^{k } \, z^{-(\lambda+1)}J_{k+1+\lambda}( z) \, C^{\lambda+1}_{k}(w)
	\end{align*}
	As $\lambda+1=m/2$, this last expression is precisely $K_{m+2}(x,y)$.
\end{proof}

\subsubsection{Even dimension}

The previous lemma allows us to move up in dimension in steps of two. We now consider only even dimension, starting with the two-dimensional case. 
An important asset in the explicit computation of a formula for the kernels in dimension $m=2$ will be the property of 4-periodicity in the index $k$, as is the case for the solutions specified in \hyperref[theo4per]{Theorem~\ref*{theo4per}}. 
In the two-dimensional case, the Gegenbauer polynomials in the kernel \eqref{kern} reduce to cosines. 
The 4-periodicity in $k$ then allows us to make explicit use of the formulas 
\begin{equation}
\label{cost}
\cos ( z \sin \theta) =  J_0(z) +   2 \sum_{n=1}^{+\infty}J_{4n}( z) \, \cos (4n \theta) +   2 \sum_{n=0}^{+\infty} J_{4n+2}( z) \, \cos \bigl((4n+2) \theta\bigr) ,
\end{equation}
\begin{equation}
\label{coss}
\cos ( z \cos \theta) =  J_0(z) +   2 \sum_{n=1}^{+\infty}J_{4n}( z) \, \cos (4n \theta) -   2 \sum_{n=0}^{+\infty} J_{4n+2}( z) \, \cos \bigl((4n+2) \theta\bigr) ,
\end{equation}
\begin{equation}
\label{sins}
\sin (  z \cos \theta)=   2 \sum_{n=0}^{+\infty} J_{4n+1}( z) \, \cos \bigl((4n+1) \theta\bigr)
-  2 \sum_{n=0}^{+\infty}  J_{4n+3}( z) \, \cos \bigl((4n+3) \theta\bigr) ,
\end{equation}
which can be found in \cite[p.\,22]{MR0010746}, formulas (1), (3) and (4). 
In this way we arrive at the following result.

\begin{theorem}\label{theoTabc2}
	In dimension $m=2$, the operator exponential
	\begin{equation*}
	T_{abc}=   \e^{i\frac{\pi}{2}  F_{abc}(\sqrt{\Om})}\e^{i \frac{\pi}{2}  ( h-1 )   }
	\end{equation*}
	with $F_{abc}(x) = a 
	+ b \,\cE_{0101} (x)  + c \,\cE_{0010} (x)$ (as specified in \hyperref[moddb]{Theorem~\ref*{moddb}}) and $a,b,c\in \{0,1,2,3\}$, 
	can be written as an integral transform whose kernel is given by
	\begin{equation}\label{kernelabc2}
	K_2(x,y) = i^a\Bigl(  \frac{1+ i^c}{2}  \cos(s)  +  i^{b+1} \sin(s)+ \frac{1- i^c}{2} \cos(t) \Bigr),
	\end{equation}
	where $s = \langle x,y\rangle$ and $t = \sqrt{\lvert x \rvert^2 \lvert y \rvert^2 - s^2}$.
\end{theorem}
\begin{proof}
	By \hyperref[sols]{Theorem~\ref*{sols}}, $T_{abc}$ satisfies the properties~{(i)}--{(iii)} and consequently, by \hyperref[intr]{Proposition~\ref*{intr}}, $T_{abc}$ can be written as an integral transform with kernel \eqref{kern}. 
	For $m=2$, 
	we have $\lambda=(m-2)/2 = 0$ and using the identity, for $w = \cos \theta$ and integer $k\geq1$,
	\[
	\lim_{\lambda \to 0} \Gamma ( \lambda ) \ C_k^\lambda ( \cos \theta) = \frac{2}{ k} \cos (k \theta), 
	\]
	(see \cite[Vol. {I}, section 3.15]{Erde}, formula (14))
	this kernel reduces to
	\begin{equation*}
	K_2(x,y) = \mu_{0,0} \, J_0(z) +   2 \sum_{k=1}^{+\infty} \mu_{0,k} \, J_{k}( z) \, \cos (k \theta).
	\end{equation*} 
	
	From \hyperref[moddb]{Theorem~\ref*{moddb}} we see that the eigenvalues of $T_{abc}$ are 4-periodic in $k$ and for $m=2$ 
	we have
	\[
	\mu_{0,0} = i^a,\quad  \mu_{0,1} = i^{a+b} i,\quad  \mu_{0,2} =i^{a +c} (-1) ,\quad  \mu_{0,3} = i^{a+b} (-i).
	\]
	This allows us to rewrite the kernel as
	\begin{align*}
	K_2(x,y) =  i^a \biggl( & J_0(z) +   2 \sum_{n=1}^{+\infty}J_{4n}( z) \, \cos (4n \theta) - i^{c}\,   2 \sum_{n=0}^{+\infty} J_{4n+2}( z) \, \cos \bigl((4n+2) \theta\bigr)  \\
	&+   i^{b+1} \Bigl(  2 \sum_{n=0}^{+\infty} J_{4n+1}( z) \, \cos  \bigl((4n+1) \theta\bigr)
	-  2 \sum_{n=0}^{+\infty}  J_{4n+3}( z) \, \cos \bigl((4n+3) \theta\bigr) \Bigr) \biggr)
	\end{align*} 
	The formulas \eqref{cost}, \eqref{coss}, \eqref{sins}, $s
	=z\cos\theta$ and $t
	= z \sin \theta$ then yield
	\begin{equation*}
	K_2(x,y) = i^a\Bigl( \frac{1}{2} \bigl(\cos(s) + \cos(t) \bigr) + i^c \frac12 \bigl(\cos(s) - \cos(t) \bigr) +  i^{b+1} \sin(s) \Bigr).
	\end{equation*}
\end{proof}

Using \hyperref[lemmak]{Lemma~\ref*{lemmak}} we now find:
\begin{theorem}\label{theoTabc}
	In even dimension $m$, the operator exponential 
	\begin{equation}\label{Tabc}
	T_{abc}=   \e^{i\frac{\pi}{2}  F_{abc}(\sqrt{\Om})}\e^{i \frac{\pi}{2}  ( h-\frac{m}{2} )   }
	\end{equation}
	with $F_{abc}(x) = a 
	+ b \,\cE_{0101} (x)  + c \,\cE_{0010} (x)$ (as specified in \hyperref[moddb]{Theorem~\ref*{moddb}}), 
	can be written as an integral transform whose kernel is given by
	\begin{equation}\label{kernelabc}
	K_m(x,y) =i^a  (-i)^{\lambda} \Bigl(  \frac{1+ i^c}{2} \bigl( \partial_s \bigr)^{\lambda}  \cos(s) +  i^{b+1} \bigl( \partial_s \bigr)^{\lambda} \sin(s)+ \frac{1- i^c}{2}  \Bigl( \partial_s -  \frac{s}{t} \   \partial_t \Bigr)^{\lambda}\cos(t)  \Bigr),
	\end{equation}
	with $\lambda=(m-2)/2$, $s = \langle x,y\rangle$ and $t = \sqrt{\lvert x \rvert^2 \lvert y \rvert^2 - s^2}$. Moreover, one has
	\begin{equation}\label{qsdf}
	\Bigl( \partial_s -  \frac{s}{t} \   \partial_t \Bigr)^{\lambda}\cos(t) =  \sqrt{\frac{\pi}{2}} \sum_{\ell =0}^{\left\lfloor  \frac{\lambda}{2} \right\rfloor} s^{\lambda-2 \ell}\    \frac{1}{2^{\ell} \ell!} \frac{\Gamma(\lambda+1)}{\Gamma(\lambda+1-2\ell)}   \frac{ {J}_{\lambda-1/2-\ell}(t)}{t^{\lambda-1/2-\ell}} .
	\end{equation}
\end{theorem}
\begin{proof}
	For $m=2$ we have $\lambda = 0$, and the expression \eqref{kernelabc} coincides with the kernel \eqref{kernelabc2} which was obtained in 
	\hyperref[theoTabc2]{Theorem~\ref*{theoTabc2}}. 
	By successive application of \hyperref[lemmak]{Lemma~\ref*{lemmak}} we have for even dimension $m\geq 2$
	\[
	K_{m} = \bigl(- i\, z^{-1} \partial_w\bigr)^{\lambda} \, K_2.
	\]
	From $s=zw$ and $t=z \sqrt{1-w^2}$, 
	we easily find
	\[
	z^{-1}  \partial_w = z^{-1} \partial_w \lbrack s\rbrack\ \partial_s +z^{-1} \partial_w \lbrack t\rbrack\ \partial_t 
	= \partial_s -  \frac{w}{\sqrt{1-w^2}} \   \partial_t  = \partial_s -  \frac{s}{t} \   \partial_t ,
	\]
	which proves \eqref{kernelabc}. 
	
	We now show \eqref{qsdf} by induction.   
	The statement holds for $m=2$ as equation \eqref{qsdf} then reduces to the identity (see \eqref{sincos1}) 
	\[
	\cos t = \sqrt{\frac{\pi }{2} }\, t^{1/2} J_{-1/2} (t)
	.\]
	Before we continue, we make a distinction between $\lambda$ even and $\lambda$ odd (or equivalently $m \equiv 2 \pmod{4}$ and $m \equiv 0 \pmod{4}$), as the upper bound of the summation in \eqref{qsdf} contains a floor function. 
	
	Consider the case $\lambda=2j$ (or thus $m=4j+2$) and assume \eqref{qsdf} holds for this $\lambda$. 
	We have, using property \eqref{ddt} of the Bessel function, 
	\begin{align*}
	\Bigl( \partial_s -  \frac{s}{t} \   \partial_t \Bigr)^{\lambda+1}\cos(t) =&\ \Bigl( \partial_s -  \frac{s}{t} \   \partial_t \Bigr) \biggl(   \sqrt{\frac{\pi}{2}} \sum_{\ell =0}^{j} s^{2j-2 \ell} \frac{1}{2^{\ell} \ell!} \frac{\Gamma(2j+1)}{\Gamma(2j+1-2\ell)} \frac{  {J}_{(m-2\ell-3)/2}( t)}{ t^{(m-2\ell-3)/2}}\biggr) \\
	=& \  \sqrt{\frac{\pi}{2}} \sum_{\ell=0}^{j-1}  s^{2j-2\ell-1}\frac{1}{2^{\ell} \ell!} \frac{\Gamma (2j+1)}{\Gamma (2j -2 \ell)}  \frac{  {J}_{(m-2\ell-3)/2}( t)}{ t^{(m-2\ell-3)/2}}\\ &+  \sqrt{\frac{\pi}{2}} \sum_{\ell=0}^{j} s^{2j-2\ell+1} \frac{1}{2^{\ell} \ell!} \frac{\Gamma (2j+1)}{\Gamma (2j +1-2 \ell)}  \frac{  {J}_{(m-2\ell-1)/2}( t)}{ t^{(m-2\ell-1)/2}}\\
	=& \  \sqrt{\frac{\pi}{2}} \sum_{\ell=1}^{j} s^{2j-2\ell+1} \frac{2 \ell}{2^{\ell} \ell!} \frac{\Gamma (2j+1)}{\Gamma (2j +2-2 \ell)} \frac{  {J}_{(m-2\ell-1)/2}( t)}{ t^{(m-2\ell-1)/2}} \\ & +   \sqrt{\frac{\pi}{2}}   \sum_{\ell=0}^{j} s^{2j-2\ell+1}  \frac{(2j +1-2 \ell)}{2^{\ell} \ell!} \frac{\Gamma (2j+1)}{\Gamma (2j +2-2 \ell)}  \frac{  {J}_{(m-2\ell-1)/2}( t)}{ t^{(m-2\ell-1)/2}}
	\\
	=&\     \sqrt{\frac{\pi}{2}} \sum_{\ell=1}^{j}  s^{2j-2\ell+1} \frac{(2j+1)}{2^{\ell} \ell!} \frac{\Gamma (2j+1)}{\Gamma (2j +2-2 \ell)}    \frac{  {J}_{(m-2\ell-1)/2}( t)}{ t^{(m-2\ell-1)/2}}  \\ & +   \sqrt{\frac{\pi}{2}}\ s^{2j+1}  \ (2j+1)  \frac{  {J}_{(m-1)/2}( t)}{ t^{(m-1)/2}}\\
	=& \  \sqrt{\frac{\pi}{2}}  \sum_{\ell=0}^{j} s^{2j-2\ell+1}  \frac{1}{2^{\ell} \ell!} \frac{\Gamma (2j+2)}{\Gamma (2j+2 -2 \ell)}  \frac{  {J}_{(m-2\ell-1)/2}( t)}{ t^{(m-2\ell-1)/2}},
	\end{align*}
	as required. 
	The inductive step in the case $\lambda$ odd is treated similarly.
\end{proof}

\begin{remark}\label{remarkbound}
	The explicit form for the kernel obtained in the preceding theorem,  
	together with the Helmholtz relations, yields its polynomial boundedness 
	similar to Lemma 5.2 and Theorem 5.3 in \cite{DBXu}, which can be proven in exactly the same fashion. 
	This ensures the corresponding integral transforms to be well-defined and continuous on $\mathcal{S}(\mathbb{R}^m)$.
\end{remark}

Next, we consider some specific cases of operators manifested in \hyperref[theoTabc]{Theorem~\ref*{theoTabc}}. 
The function $F_{abc}$ in \eqref{Tabc} contains three parameters, each having four possible values (up to modulo 4 congruence). The role of the parameter $a$ is but a scalar multiplicative factor, so we take $a=0$ in the following.

Putting $c=0$, \eqref{Tabc} reduces to the operator exponential 
\begin{equation}\label{Tb}
T_b =  \e^{i\frac{\pi}{2} b{\Om}
}\e^{i \frac{\pi}{2}  ( h-\frac{m}{2} )   },
\end{equation}
where we used $\cE_{0101}(x) = E_1(x)= x^2$. 
When written as an integral transform we find that its kernel is given by 
\begin{equation}\label{kernelb}
K_m(x,y) =  \bigl(-i\, \partial_s \bigr)^{\lambda}\bigl(    \cos(s) +  i^{b+1} \sin(s) \bigr)= i^{b\lambda^2} \cos(\langle x, y\rangle)+ i^{b(\lambda+1)^2+1} \sin(\langle x, y\rangle).
\end{equation}
Here we distinguish 4 possible scenarios for the value of $b$:
\begin{itemize}
	\item
	For $b\equiv 0 \pmod{4}$, the operator exponential \eqref{Tb} is precisely the classical Fourier transform and \eqref{kernelb} indeed gives
	$
	K_m(x,y) =\cos(\langle x, y\rangle)+ i\, \sin(\langle x, y\rangle) =  e^{i\langle x, y\rangle}
	$.
	\item
	Taking $b \equiv 2 \pmod{4}$ and multiplying \eqref{kernelb} by $ \e^{{i\pi}\lambda}$, we get the kernel
	$\cos(\langle x, y\rangle)- i\, \sin(\langle x, y\rangle) =  e^{-i\langle x, y\rangle}$ 
	of the inverse Fourier transform and hence
	\[
	\cF^{-1} =  \e^{{i\pi}\lambda} \e^{{i\pi}\Om}
	\e^{i \frac{\pi}{2}  ( h-\frac{m}{2} )   }= 
	\e^{-i \frac{\pi}{2}  ( h-\frac{m}{2} )   }.
	\]
	As $\cF^{-1} = \cF^3$, we also have
	\[
	\cF^2= \e^{{i\pi}(\Om+\lambda)}= \e^{i \pi  ( h-\frac{m}{2} )   } .
	\]
	\item 
	The other 
	cases for the value of $b$, namely $b\equiv 1 \pmod{4}$ and $b\equiv 3 \pmod{4}$, give rise to another pair of interesting transforms. Taking $b=2\lambda +1$ and multiplying by an appropriate multiplicative factor, we find the kernel $ \cos(\langle x, y\rangle)- \sin(\langle x, y\rangle)$, corresponding to the 
	operator exponential 
	\begin{equation*}
	\e^{i\frac{\pi}{2}\lambda^2} \e^{i\frac{\pi}{2} (2\lambda +1){\Om}}\e^{i \frac{\pi}{2}  ( h-\frac{m}{2} )   }.
	\end{equation*}
	Furthermore, the operator exponential 
	\begin{equation*}
	\e^{i\frac{\pi}{2} (\lambda^2+2\lambda)} \e^{i\frac{\pi}{2} (2\lambda -1){\Om}}\e^{i \frac{\pi}{2}  ( h-\frac{m}{2} )   }
	\end{equation*}
	can be written as an integral transform whose kernel is given by
	$
	\cos(\langle x, y\rangle)+ \sin(\langle x, y\rangle)
	$. 
	This is the cosine-and-sine or Hartley kernel of the integral transform known as the Hartley transform. 
	The Hartley transform is a real linear operator that is symmetric and Hermitian \cite{Brace,Bracewell}. Moreover, it is a unitary operator that is its own inverse. 
\end{itemize}
We summarize this in the following table
\[
\begin{array}{cc}
\tT = T \circ \cF^{-1} & K_m(x,y) \\ \hline
\e^{i\frac{\pi}{2} b{\Om}}&  i^{b\lambda^2} \cos(\langle x, y\rangle)+ i^{b(\lambda+1)^2+1} \sin(\langle x, y\rangle)\\
1 &  \cos(\langle x, y\rangle)+ i  \sin(\langle x, y\rangle) \\
\e^{{i\pi}\lambda} \e^{{i\pi}\Om}&\cos(\langle x, y\rangle)- i \sin(\langle x, y\rangle)\\
\e^{i\frac{\pi}{2}\lambda^2} \e^{i\frac{\pi}{2} (2\lambda +1){\Om}}& \cos(\langle x, y\rangle)- \sin(\langle x, y\rangle) \\
\e^{i\frac{\pi}{2} (\lambda^2+2\lambda)} \e^{i\frac{\pi}{2} (2\lambda -1){\Om}} &  \cos(\langle x, y\rangle)+ \sin(\langle x, y\rangle)
\end{array}
\]
Note that the kernel of all of these integral transforms is of the form 
\[
K_m(x,y) =
\cos(\langle x, y\rangle) + i^{d} \sin(\langle x, y\rangle)
\]
for some integer $d$. As $ \cos(\langle x, y\rangle)$ and $\sin(\langle x, y\rangle)$ are given by the real and imaginary parts of the Fourier transform,  an integral transform with such a kernel coincides with
\[
\frac12 \Bigl( (\cF  f)(y) + (\cF  f)(-y )\Bigr) + i^{d-1} \frac12 \Bigl(  (\cF  f)(y) - (\cF  f)(-y) \Bigr)
\rlap{.}
\]

Finally, we consider one more special instance. Putting $c=2$ (and again $a=0$), \eqref{kernelabc} reduces to the kernel
\[
K_m(x,y) =    i^{b+1} i^{-\lambda} \sin\Bigl(s+\lambda\frac{\pi}{2} \Bigr)+  i^{-\lambda}    \sqrt{\frac{\pi}{2}} \sum_{\ell =0}^{\left\lfloor  \frac{m-2}{4} \right\rfloor} s^{\frac{m}{2}-1-2 \ell}\    \frac{1}{2^{\ell} \ell!} \frac{\Gamma(\frac{m}{2})}{\Gamma(\frac{m}{2}-2\ell)}   \frac{ {J}_{(m-2\ell-3)/2}(t)}{t^{(m-2\ell-3)/2}}\rlap{,}
\]
corresponding to the operator exponential
\begin{equation*}
T_b =  \e^{i\frac{\pi}{2} ( \frac{3}{2}{\Om}^2 - \frac{3}{2} {\Om}+ b{\Om})
}\e^{i \frac{\pi}{2}  ( h-\frac{m}{2} )   }.
\end{equation*}
For $m=2$ this kernel becomes
\[
K_2(x,y) =   i^{b+1} \sin(\langle x,y\rangle)+  \cos\Bigl( \sqrt{\lvert x \rvert^2 \lvert y \rvert^2 - \langle x,y\rangle^2}\Bigr) .
\]

\subsubsection{Odd dimension}

While \hyperref[lemmak]{Lemma~\ref*{lemmak}} remains valid for odd dimension $m$, we have no specific formulas for dimension $m=3$ comparable to those we used for dimension $m=2$, i.e.~formulas (\ref{cost})--(\ref{sins}). 
We do have another way to obtain a closed formula for a kernel of the form \eqref{kern}. This other approach holds for a more restricted class of operators (for even dimension this class is already included in \hyperref[theoTabc]{Theorem~\ref*{theoTabc}}) but it has the advantage that we can also use this in odd dimension. 

\begin{lemma}\label{lemmasc}
	If the action of an operator $T$ on the Hermite basis $\{\phi_{j,k,\ell}\}$ is given by 
	\[ T \, \phi_{j,k,\ell} = \mu_{j,k} \, \phi_{j,k,\ell}, \]
	with eigenvalues $\mu_{j,k}\in \mC$ 
	that satisfy
	$
	\mu_{j+1,k} =- \mu_{j,k}$ and $
	\mu_{j,k+2} =- \mu_{j,k}
	$ 
	for $j, k \in \mZ_{\geq0}$. 
	Then, on $\{\phi_{j,k,\ell}\}$, $T$ can be written as an integral transform 
	\[
	(T f)  (y) = \frac{1}{(2\pi)^{m/2}} \int_{\mathbb{R}^m}\! K_m(x,y) \, f(x) \, \mathrm{ d}x
	\]
	with kernel 
	\begin{equation}\label{kerncs}
	K_m(x,y) =\mu_{0,0} \cos(\langle x, y\rangle)+\mu_{0,1} \sin(\langle x, y\rangle).
	\end{equation}
\end{lemma}
\begin{proof}Let $T$ be as specified in the lemma. From the conditions on its eigenvalues we know that the spectrum of $T$ is completely determined by its eigenvalues $\mu_{0,0}$ and $\mu_{0,1}$, while the others follows from 
	\[
	\mu_{j,2n} =(-1)^j(-1)^n \mu_{0,0},
	\qquad
	\mu_{j,2n+1} =(-1)^j (-1)^n \mu_{0,1}.
	\]
	From the proof of \hyperref[intr]{Proposition~\ref*{intr}} we know that
	the integral transform 
	\[
	(T f)  (y) = \frac{1}{(2\pi)^{m/2}} \int_{\mathbb{R}^m}\! K_m(x,y) \, f(x) \, \mathrm{ d}x
	\]
	with kernel $K_m(x,y)$, given by \eqref{kern},
	will have the same eigenvalue spectrum as that of $T$. 
	
	Distinguishing between even and odd values of $k$, the series in $K_m(x,y)$ becomes
	\begin{align*}
	K_m(x,y) = & \  \mu_{0,0}   \, {2}^{\lambda}\, \Gamma ( \lambda )    \sum_{n=0}^{+\infty} (2n+\lambda)\,  (-1)^n  \, z^{-\lambda}J_{2n+\lambda}( z) \,C^{\lambda}_{2n}(w) \\ 
	&+  \mu_{0,1}  \,  {2}^{\lambda}\, \Gamma ( \lambda )    \sum_{n=0}^{+\infty} (2n+1+\lambda)\, (-1)^n \, z^{-\lambda}J_{2n+1+\lambda}( z) \,C^{\lambda}_{2n+1}(w).
	\end{align*}
	The desired result now follows from formulas (44) and (45) of \cite[Vol.~{II}, section 7.15]{Erde}.
\end{proof}
Note that the condition $
\mu_{j,k+2} =- \mu_{j,k}
$ 
for $j, k \in \mZ_{\geq0}$ immediately implies the eigenvalues being 4-periodic in the index $k$. 
In even dimension, one easily verifies that the operator exponential $T_b$ given by \eqref{Tb} 
satisfies the conditions of \hyperref[lemmasc]{Lemma~\ref*{lemmasc}}.
Indeed, we have already obtained its kernel to be \eqref{kernelb}, which is of the form \eqref{kerncs}. 

Now, in the odd dimensional case, we find an operator with suitable eigenvalues by constructing a function 
that alternates between two values modulo 4, as is the case for the function in the exponent of $T_b$ in even dimension. 
The function
\[
F_{ab}(x) =a+b\,\bigl( \cD_{0110}(x)  + \cD_{0011}(x)  +2\cD_{0010}(x)  \bigr) \rlap{.}
\]
meets this requirement as for $k\in\mathbb{Z}_{\geq0}$, $F_{ab}\bigl( 2k\!+\!\tfrac12\bigr) \equiv a$ and $F_{ab}\bigl( 2k\!+\!1\!+\!\tfrac12\bigr) \equiv a+b$.
This function allows us to state the following theorem, where we have plugged in values for $a$ and $b$ to yield the simplest form for the kernel function.

\begin{theorem}\label{theoTd}
	In odd dimension, let $T_d$ be the operator exponential  
	\begin{equation*}
	T_d =  \e^{i\frac{\pi}{2}d F(\sqrt{\Om})
	}\e^{i \frac{\pi}{2}  ( h-\frac{m}{2} )   }
	\end{equation*}
	with $d \in \{0,1,2,3\}$ and 
	\[
	F\bigl(\sqrt{\Om}\bigr) =(m-2)\Bigl( \cD_{0110}\bigl(\sqrt{\Om}\bigr)  + \cD_{0011}\bigl(\sqrt{\Om}\bigr)  +2\cD_{0010}\bigl(\sqrt{\Om}\bigr)  \Bigr)+  \Bigl( \frac{m+1}{2} \Bigr)^2 \rlap{.}
	\]
	Then $T_d$ can be written as an integral transform 
	\[
	(T f)  (y) = \frac{1}{(2\pi)^{m/2}} \int_{\mathbb{R}^m}\! K_m(x,y) \, f(x) \, \mathrm{ d}x
	\]
	whose kernel is given by
	\begin{equation*}
	K_m(x,y) = \cos\bigl(\langle x, y\rangle\bigr)+i^{d+1} \sin\bigl(\langle x, y\rangle\bigr)\rlap{.}
	\end{equation*}
\end{theorem}
\begin{proof}
	The operator $T_d$ is 4-periodic in the index $k$ by \hyperref[theo4per]{Theorem~\ref*{theo4per}} and its eigenvalues are given by
	\[
	\e^{i\frac{\pi}{2}dF(\sqrt\Om) }  \e^{i\frac{\pi}{2} ( h-\frac{m}{2} ) }\, \phi_{j,k,\ell}= \e^{i\frac{\pi}{2}dF(k+\lambda) }  \e^{i\frac{\pi}{2}(2j+ k) }\, \phi_{j,k,\ell}\rlap{\,.}
	\]
	The first four eigenvalues with $j=0$ are
	\[
	\mu_{0,0} = 1 ,\quad\mu_{0,1} = i^{d+1},\quad\mu_{0,2} = -1,\quad\mu_{0,3} = -i^{d+1},
	\]
	Using also \hyperref[eigenv]{Lemma~\ref*{eigenv}}, we see that $T_d$ satisfies the conditions of \hyperref[lemmasc]{Lemma~\ref*{lemmasc}} and thus find the desired kernel.
\end{proof}


\subsection{Uncertainty principle}
\label{ssUp}

As an application of our previous results, we show how to obtain generalized uncertainty principles (following the same strategy as developed for the Dunkl transform in \cite{MR1698045,MR1818904} and later generalized and streamlined in \cite{Orsted2}) for any continuous integral transform
\[
(T f) (y) = \int_{\mR^m} K(x,y)f(x) \,\mathrm{d}x
\] 
on $\cS(\mR^m)$ that satisfies all the properties of \hyperref[theo4per]{Theorem~\ref*{theo4per}}. Hence, it applies in particular to the operators from e.g.~\hyperref[theoTabc]{Theorem~\ref*{theoTabc}} and \hyperref[theoTd]{Theorem~\ref*{theoTd}} as they have polynomially bounded kernels (see \hyperref[remarkbound]{Remark~\ref*{remarkbound}}). 
By \hyperref[unitary]{Proposition~\ref*{unitary}}, such a $T$ has a unitary extension to $L^2(\mR^m)$ and 
we may now establish the following lemma.

\begin{lemma}
	\label{Heis_lem}
	If the continuous integral transform $T$ on $\cS(\mR^m)$ satisfies the properties {(i)}--{(iii)}, 
	then for its unitary extension to $L^2(\mR^m)$ the following inequality holds:
	\begin{equation}
	\label{Heis}
	\bigl\| \lvert x\rvert f \bigr\|^2 + \bigl\| \lvert x\rvert T(f)  \bigr\|^2 \geq m \| f \|^2, \qquad f\in \cS(\mR^m),
	\end{equation}
	with $\| \cdot \|$ the $L^2$ norm. 
	The inequality becomes an equality if and only if $f = \alpha e^{-|x|^2/2}$ with $\alpha \in \mR$.
\end{lemma}

\begin{proof}
	We can compute that for $f \in \cS(\mR^m)$
	\begin{align*}
	\bigl\| \lvert x\rvert T(f)  \bigr\|^2  & = \langle |x| T(f), |x| T(f) \rangle\\
	&= \langle |x|^2 T(f), T(f) \rangle\\
	&= - \langle T( \Delta_x f), T(f) \rangle\\
	&=  - \langle\Delta_x f, f \rangle
	\end{align*}
	where we used the Helmholtz relations and the unitarity of $T$. Using this result, the left-hand side of (\ref{Heis}) equals
	\begin{align*}
	\bigl\| \lvert x\rvert f \bigr\|^2+ \bigl\| \lvert x\rvert T(f)  \bigr\|^2 &= \langle |x|^2 f, f \rangle - \langle\Delta_x f, f \rangle \\
	&=  \langle ( |x|^2 - \Delta_x) f, f \rangle\\
	&=  2 \langle h f, f \rangle\\
	& \geq m \langle f,f \rangle
	\end{align*}
	as the smallest eigenvalue of $h$ is $m/2$.
	The equality then follows as this minimal eigenvalue corresponds to the ground state $f = \alpha e^{-|x|^2/2}$.
\end{proof}

Subsequently we establish the following generalized uncertainty principle.

\begin{theorem}\label{theoUP}
	Let
	\[
	(T f) (y) = \int_{\mR^m} K(x,y)f(x)  \, \mathrm{ d} x
	\] 
	be a continuous integral transform on $\cS(\mR^m)$ that satisfies the properties {(i)}--{(iii)}, 
	then one has the following uncertainty principle:
	\[
	\bigl\| \lvert x\rvert f \bigr\|  \cdot\bigl \| \lvert x\rvert  T(f )\bigr\| \geq \frac{m}{2} \| f \|^2
	\]
	for $f \in \cS(\mR^m)$.
\end{theorem}

\begin{proof}
	Put, for $c >0$, $f_c(x) = f(cx)$. An easy computation then shows that
	\[
	\bigl\| \lvert x\rvert f_c \bigl \|^2 = c^{-m-2} \bigl\| \lvert x\rvert f \bigl \|^2
	\]
	and similarly
	\[
	\| f_c \|^2 = c^{-m} \|  f \|^2.
	\]
	We also have, using the homogeneity of the kernel as given in (\ref{Hom_K}),
	\begin{align*}
	(Tf_c) (y) & = \int_{\mR^m} K(cx,y/c)f( c x) \, \mathrm{ d} x\\
	& =  c^{-m} \int_{\mR^m} K(x,y/c)f( x) \, \mathrm{ d} x\\
	& = c^{-m} (Tf) (y/c).
	\end{align*}
	As a consequence we have
	\begin{align*}
	\bigl\| \lvert x\rvert T(f_c) \bigl \|^2 & = \bigl\|  c^{-m}\lvert x\rvert (Tf)(y/c) \bigl \|^2\\
	&= c^{-2m} \int_{\mR^m} |x|^2 |(Tf)(x/c)|^2 \, \mathrm{ d} x\\
	& =c^{-m+2} \bigl\| \lvert x\rvert T(f) \bigl \|^2.
	\end{align*}
	Now substitute $f_c$ for $f$ in Lemma \ref{Heis_lem}, and apply the previously established relations. This yields:
	\[
	c^{-2} \bigl\| \lvert x\rvert f \bigl \|^2 + c^2 \bigl\| \lvert x\rvert T(f)\bigl \|^2 \geq m \| f \|^2.
	\]
	Finally put 
	\[
	c = \sqrt{\frac{\bigl\| \lvert x\rvert f \bigl \|}{\bigl\| \lvert x\rvert T(f) \bigl \|}}
	\]
	and the theorem follows.
\end{proof}

\begin{remark}
	More specialized uncertainty principles can be developed following a strategy similar to the one used in e.g.~\cite{Joh}.
\end{remark}


\section{Fourier transforms in Clifford analysis}
\label{sec:3}

In the previous section we  
obtained a class of operators satisfying the set of properties {(i)}--{(iii)}.
We did this by capitalizing on the relation between these properties, the classical Fourier transform and a realization of the Lie algebra $\mathfrak{sl}(2)$. 
Now, we wish to expand this train of thought to a broader setting. 
The aforementioned operator realization of $\mathfrak{sl}(2)$ possesses a natural generalization
to the Lie superalgebra $\mathfrak{osp}(1|2)$. 
For this reason, we turn our attention to the context of Clifford analysis and aim to find solutions by applying the techniques developed in the previous section.

We start by giving a brief overview of the framework of Clifford analysis  (see e.g.~\cite{MR697564,DSS}), a higher dimensional function theory where functions take on values in a Clifford algebra. 
The orthogonal Clifford algebra $\cC l_{m}$ is generated by the canonical basis $\{ \, e_{i}\mid  i = 1, \ldots, m\,\}$ of $\mR^m$ under the relations:
\[
\begin{array}{l}
e_i e_j + e_j e_i = 0 \qquad (i\neq j)\\
e_i^2 = -1.
\end{array}
\]
This algebra has dimension $2^m$ as a vector space over $\mR^m$, and we have \[\cC l_{m} = \bigoplus_{k=0}^{m} 
\mbox{span} \{ \,e_{i_{1}} e_{i_{2}} \dotsm e_{i_{k}} \mid 1 \leq i_{1} < \dotsb < i_{k} \leq m \,\}.
\]
The empty product ($k=0$) is defined as the multiplicative identity element. 

Functions taking values in $\cC l_{m}$ can be decomposed as 
\[
f = f_{0} + \sum_{i=1}^{m} e_{i}f_{i} + \sum_{i< j} e_{i} e_{j} f_{ij} + \dotsb + e_{1} \dotsm e_{m} f_{1 \dotso m} 
\]
with  $f_{0}, f_{i}, f_{ij}, \ldots, f_{1 \ldots m}$ all real-valued functions on $\mR^{m}$. 
We identify the point $x = (x_{1}, \dotsc, x_{m})$ in $\mR^{m}$ with the vector variable $\ux$ given by 
\[
\ux = \sum_{j=1}^{m}  e_{j}x_{j}.
\]
The Clifford product of two vectors splits into a scalar part and a bivector part, given by respectively minus the inner product of the two vectors and the outer product or wedge product: 
\[
\ux \, \uy =- \langle x, y \rangle + \ux \wedge \uy,
\]
with $- \langle x, y \rangle = -\sum_{j=1}^{m} x_{j} y_{j} = \frac{1}{2} (\ux \, \uy + \uy  \, \ux)$ and $
\ux \wedge \uy = \sum_{j<k} e_{jk} (x_{j} y_{k} - x_{k}y_{j}) = \frac{1}{2} (\ux \, \uy - \uy \,  \ux).
$ 

Furthermore, we introduce a first order vector differential operator by
\[
\upx = \sum_{j=1}^{m} \partial_{x_{j}} e_{j}.
\]
This operator is the so-called Dirac operator. Together with the vector variable they satisfy the relations
\[
\upx^{2} =  -\Delta_x,\qquad
\ux^{2} = -|x|^{2},\qquad
\{\ux, \upx\} = -2\mE_x -m,
\]
where $\{a,b\}=ab+ba$, and hence they generate a realization of the Lie superalgebra $\mathfrak{osp}(1|2)$, which contains the Lie algebra $\mathfrak{sl}(2) =\mathrm{span}\big \{\Delta_x,\lvert x \rvert^2,\big [\Delta_x,\lvert x \rvert^2 \big]\big \}
$ as its even part \cite{MR1773773}.

As the Clifford-valued operators $\upx$ and $\ux$ factorize the operators $-\Delta_x$ and $ -|x|^{2}$, they allow us to refine the Helmholtz relations 
to what we call the Clifford-Helmholtz relations
\begin{flalign*}
T \circ \partial_{\ux} &= - i\, \uy \circ T \\
T \circ \ux &=  - i \, \partial_{\uy} \circ T
\end{flalign*}
Every operator that satisfies this system will by definition also satisfy the Helmholtz relations. 
They form an intermediate step in the generalization of the properties 
of the Fourier transform:
\[
\biggl\{\begin{array}{rl}
T \circ \partial_{x_{j}}= - i \,y_{j}  \circ T\\
T\circ x_{j}= - i \,  \partial_{y_{j}}\circ T
\end{array}  \implies\biggl\{
\begin{array}{ll}
T \circ \partial_{\ux} = - i\, \uy \circ T \\
T \circ \ux =  - i \, \partial_{\uy} \circ T
\end{array} \implies \biggl\{
\begin{array}{ll}
T \circ   \Delta_x    =-\lvert y \rvert^2 \circ  T\\
T \circ  \lvert x \rvert^2   = - \Delta_y\circ  T
\end{array}
\]
where in the first we assume it must hold for all $j \in \{ 1,\dotsc,m\}$.

In the previous section we made extensive use of the eigenfunction basis of $\mathcal{S}(\mathbb{R}^m) \subset  L^{2}(\mathbb{R}^m)$ given by the functions
\eqref{basis}. 
In the framework of Clifford analysis the relevant function space is the space $\mathcal{S}(\mathbb{R}^m) \otimes \cC l_{m}$ which decomposes under the action of the dual pair $(\mathfrak{osp}(1|2),Spin(m) )$. 
This action leads to the following important basis 
(see \cite{MR926831})
\begin{align} \label{basisp}
\begin{split}
\psi_{2p,k,\ell}(x) &\defeq  2^p  p!\, L_{p}^{\frac{m}{2}+k-1}(|x|^{2}) \, M_{k}^{(\ell)} \, \e^{-|x|^{2}/2},\\
\psi_{2p+1,k,\ell}(x) &\defeq 2^p p!\sqrt{2}   \,  L_{p}^{\frac{m}{2}+k}(|x|^{2}) \, \ux \, M_{k}^{(\ell)} \, \e^{-|x|^{2}/2},
\end{split}
\end{align}
where $p,k \in \mathbb{Z}_{\geq0}$ and $\{\,M_k^{(\ell)} \mid \ell= 1, . . . , \dim(\cM_k)\, \}$ is a basis for $\cM_k$, 
the space of spherical monogenics of degree $k$, 
i.e.~homogeneous polynomial null-solutions of the Dirac operator of degree $k$. 
It is clear that every spherical monogenic is a spherical harmonic, indeed we have 
$\cH_k\otimes \cC l_m = \cM_k \oplus \ux  \cM_{k-1}$. 

The action of the Fourier transform on the eigenfunctions $\{\psi_{j,k,\ell}\}$ is given by
\[
\cF \psi_{j,k,\ell} = \e^{i\frac{\pi}{2}(j+k)} \, \psi_{j,k,\ell} = i^{j+k} \, \psi_{j,k,\ell}.
\]
Moreover, the functions $\{\psi_{j,k,\ell}\}$ have the following important property (see e.g.~\cite{MR926831}), comparable to property \eqref{ephi}, 
\begin{equation}\label{bpsi}
\psi_{j,k,\ell}(x) =\left(\frac{\sqrt2}{2}\left(\ux - \partial_{\ux}\right)\right)^j M_{k}^{(\ell)} \, \e^{-|x|^{2}/2}\rlap{\,.}
\end{equation}

Now, our aim is to find all operators 
$
T \colon 
\mathcal{S}(\mathbb{R}^m) \otimes \cC l_{m} \to \mathcal{S}(\mathbb{R}^m) \otimes \cC l_{m}
$
that satisfy
\begin{itemize}
	\item[(I)] the Clifford-Helmholtz relations
	\begin{flalign*}
	&T \circ \partial_{\ux} = - i\, \uy \circ T &&\\
	&T \circ \ux =  - i \, \partial_{\uy} \circ T&&
	\end{flalign*}
	
	\item[(II)] $ T \, \psi_{j,k,\ell} = \mu_{j,k} \, \psi_{j,k,\ell} \qquad \text{with }\mu_{j,k}\in \mC$
	
	\item[(III)] $T^4 = \operatorname{id}$
\end{itemize}

In line with \hyperref[sec:2]{Section~\ref*{sec:2}} we first introduce some suitable linear combinations of the operators of interest, namely $\ux$ and $\upx$.
Hereto put
\begin{equation*}
b^+ = \frac{\sqrt2}{2}\bigl(\ux - \partial_{\ux}\bigr) \quad\text{and}\quad
b^- =- \frac{\sqrt2}{2}\bigl(\ux + \partial_{\ux}\bigr).
\end{equation*}
They satisfy the relations
\begin{equation}\label{bbbb}
\big \lbrack \{b^-,b^+\}, b^{\pm} \big\rbrack=\pm 2 b^{\pm},
\end{equation}
and hence also generate a realization of the Lie superalgebra $\mathfrak{osp}(1|2)$, \cite{LSV}. 
This Lie superalgebra contains an even subalgebra isomorphic with $\mathfrak{sl}(2)$ generated by 
the even (or ``bosonic'') elements
\[
h = \frac12 \{b^-,b^+\}
, \qquad
e = \frac14\{b^+,b^+\}
,\qquad
f= -\frac14 \{b^-,b^-\}
\rlap{.}
\]
These are precisely the operators \eqref{hef} we considered in the previous section. 

Working in this realization of $\mathfrak{osp}(1|2)$ the Clifford-Helmholtz relations 
translate to the relations
\begin{align}
\begin{split}\label{Tbb}
T \circ b^+  =\phantom{-}  i \, b^+\circ T& \\
T \circ b^-  =- i \, b^-\circ T &,
\end{split}
\end{align}
and property \eqref{bpsi} can be written more compactly as
$
\psi_{j,k,\ell}(x) =\left(b^+\right)^j M_{k}^{(\ell)} \, \e^{-|x|^{2}/2}.
$

Moreover, in this realization the operators act on the eigenfunctions \eqref{basisp} in a nice way. 
We have
\begin{equation}\label{bppsi}
b^+ \, \psi_{j,k,\ell}= \, \psi_{j+1,k,\ell}
\end{equation}
and for integer $p$,
\[
b^-  \, \psi_{2p,k,\ell} = 2p \, \psi_{2p-1,k,\ell},\qquad b^-  \, \psi_{2p+1,k,\ell} = (2p+m+2k)\, \psi_{2p,k,\ell}.
\]
The action of the other operators is then as follows
\[
h \, \psi_{j,k,\ell}=(j+k+{\tfrac{m}{2}})\, \psi_{j,k,\ell},\qquad
e \, \psi_{j,k,\ell}=\frac{1}{2}\, \psi_{j+2,k,\ell}
\]
and, again for integer $p$,
\[
f \,\psi_{2p,k,\ell}=-p\, (2p-2+m+2k)\, \psi_{2p-2,k,\ell},\qquad
f \, \psi_{2p+1,k,\ell}= -p\, (2p+m+2k)\, \psi_{2p-1,k,\ell}.
\]
Note that $h$ again acts diagonally on $\psi_{j,k,\ell} $.
\begin{remark}
	For every $k\in \mZ_{\geq0}$ and $\ell\in \{ 1, \dotsc, \dim(\cM_k)\}$, the set $\{\, \psi_{j,k,\ell} \mid j \in \mZ_{\geq0}  \, \}$ forms a basis for the
	irreducible representation of the Lie superalgebra $\mathfrak{osp}(1|2)$ 
	with lowest weight $k+m/2$. This representation is a direct sum of two positive discrete series representations of 
	$\mathfrak{su}_{1,1}$. 
	
	For every $j,k\in \mZ_{\geq0}$ the set $\{\, \psi_{j,k,\ell} \mid \ell= 1, . . . , \dim(\cM_k)  \, \}$ forms a basis for an irreducible spinor representation of $Spin(m)$, when restricting the values $\cC l_{m}$ to a spinor space.
\end{remark}

Similar to the results obtained in the harmonic case, we immediately have some important consequences for an operator $T$ that satisfies the properties {(I)}--{(III)}.
\begin{lemma}\label{eigenvC}
	Let $T$ be an operator satisfying the properties {(I)}--{(III)}. There are only four possible values for the eigenvalues $
	\mu_{j,k}$ of the operator $T$, namely 
	$\mu_{j,k}\in\{1,i,-1,-i\}.
	$
	Moreover, the spectrum of eigenvalues is completely determined by the eigenvalues $\mu_{0,k}$ for $k\in \mZ_{\geq0}$; the other eigenvalues for $j>0$ follow from the relation
	\begin{equation}\label{mujkC}
	\mu_{j,k}  = i^{\,j}\,\mu_{0,k}.
	\end{equation}
\end{lemma}
\begin{proof}
	Property {(III)} necessitates that the eigenvalues $\mu_{j,k}$ of $T$ satisfy $(\mu_{j,k})^4=1$ and thus are integer powers of $i$. 
	Using property {(II)} and the relations  \eqref{bpsi} and \eqref{bppsi}
	, we find
	\[
	\mu_{j+1,k} \, \psi_{j+1,k,\ell} = T\, \psi_{j+1,k,\ell} = T \circ b^+ \, \psi_{j,k,\ell} =  i \, b^+ \circ T  \, \psi_{j,k,\ell} = i \, b^+\,  \mu_{j,k}  \, \psi_{j,k,\ell} = i \,\mu_{j,k} \, \psi_{j+1,k,\ell}.
	\]
	The relation \eqref{mujkC} now follows from subsequent application of $ \mu_{j+1,k} = i\,  \mu_{j,k}$.
\end{proof}

\begin{proposition}\label{intrC}
	Let $
	T \colon 
	\mathcal{S}(\mathbb{R}^m)  \otimes \cC l_{m} \to \mathcal{S}(\mathbb{R}^m)  \otimes \cC l_{m}
	$ be an operator that satisfies the properties {(I)} and {(II)}. 
	Then, on the basis $\{\psi_{j,k,\ell}\}$, the operator $T$ 
	coincides with the integral transform 
	\[
	(T f)  (y) = \frac{1}{(2\pi)^{m/2}} \int_{\mathbb{R}^m}\! K_m(x,y) \, f(x) \, \mathrm{ d}x
	\]
	with
	\begin{equation}\label{kernelC}
	K_m(x,y) = A(w,z) + (\ux \wedge \uy) \  B(w,z)
	\end{equation}
	where
	\begin{align}
	\label{seriesconv}
	\begin{split}
	A(w,z) &=  {2}^{\lambda}\, \Gamma ( \lambda )    \sum_{k=0}^{+\infty} \frac{1}{2} \big(i \, k \, \mu_{0,k-1}+(k+2\lambda) \mu_{0,k}\big)\ z^{-\lambda}J_{k+\lambda}( z) C^{\lambda}_{k}(w)\\
	B(w,z) &=  2^{\lambda+1}\, \Gamma ( \lambda +1 )    \sum_{k=1}^{+\infty}  \frac{ 1}{2}\big(i \, \mu_{0,k-1}- \mu_{0,k}\big) \ z^{-\lambda-1}J_{k+\lambda}( z) C^{\lambda+1}_{k-1}(w),
	\end{split}
	\end{align}
	Here, the notations $\lambda=(m-2)/2$, $z=|x |  |y|$ and $w = \langle x,y \rangle/z$ 
	are used. 
\end{proposition}
\begin{proof}We refer to \hyperref[eigenvalues]{Theorem~\ref*{eigenvalues}} in Appendix~\ref{secA3} for the full proof of this result.
\end{proof}

In the following result, $\mathrm{Spin}(m)$ denotes the spin group, which is a subgroup of the Clifford algebra $\cC l_{m}$ that is a double cover of the special orthogonal group $\mathrm{SO}(m)$.

\begin{lemma}\label{homogenityC}
	The kernel satisfies
	\begin{align*}
	K_m(x,cy) &= K_m(cx,y), \qquad \forall c \in \mR \\
	K_m(x,Ay) &= K_m(Ax,y), \qquad \forall A \in \mathrm{Spin}(m)
	\end{align*}
\end{lemma}
\begin{proof}
	This follows from the explicit formulas (\ref{kernelC})--(\ref{seriesconv}) for $K_m(x,y)$.
\end{proof}

\begin{proposition}\label{unitaryC}
	A continuous operator $T\colon 
	\mathcal{S}(\mathbb{R}^m)\otimes \cC l_{m}  \to \mathcal{S}(\mathbb{R}^m)\otimes \cC l_{m} $ satisfying {(I)}--{(III)} has a unitary extension to 
	$L^{2}(\mathbb{R}^m) \otimes \cC l_{m}$.
\end{proposition}
\begin{proof}
	This result follows from the fact that $\mathcal{S}(\mathbb{R}^m) \otimes \cC l_{m}$ is dense in $L^{2}(\mathbb{R}^m) \otimes \cC l_{m}$ and that all eigenvalues have unit norm.
\end{proof}

We also have an uncertainty principle, using the inner product \eqref{IPC} for $L^{2}(\mathbb{R}^m) \otimes \cC l_{m}$:
\begin{theorem}\label{theoUPC}
	Let
	\[
	(T f) (y) = \int_{\mR^m} K(x,y)f(x)  \, \mathrm{ d} x
	\] 
	be a continuous integral transform on $\cS(\mR^m)\otimes \cC l_{m}$ that satisfies the properties {(I)}--{(III)}. Then one has the following uncertainty principle:
	\[
	\bigl \lVert \ux f \bigr\rVert_2 \cdot \bigl\lVert \ux T(f)\bigr\rVert_2 \geq \frac{m}{2} \bigl(\lVert f\rVert_2\bigr)^2
	\]
	for $f \in \cS(\mR^m)\otimes \cC l_{m}$.
\end{theorem}
\begin{proof}
	This follows by the same reasoning as used to prove \hyperref[theoUP]{Theorem~\ref*{theoUP}}, now using \hyperref[homogenityC]{Lemma~\ref*{homogenityC}} and noting that the Clifford-Helmholtz relations imply the regular Helmholtz relations.
\end{proof}

\begin{remark}
	In the special case of the Clifford-Fourier transform of \cite{MR2190678, MR2283868, DBNS, DBXu} more specialized uncertainty principles can be developed following the strategy of \cite{GJ}.
\end{remark}

Now, to determine operators that satisfy properties {(I)}--{(III)}, 
we proceed in the same way as we did in the previous section for the harmonic case.
We start by decomposing $T$ as 
$
T = \tT \circ \cF,
$ 
with again $\tT \defeq T \circ \cF^{-1}$. 
From \eqref{Tbb} and property {(III)}, we have the following conditions for the operator $\tT$: $\tT$ has to commute with $b^{ \pm}$ 
and
$
\tT^4 =  \operatorname{id}
$. 
Therefore, we look at the universal enveloping superalgebra of the Lie superalgebra $\mathfrak{osp}(1|2)$, denoted by $ \mathcal{U} \big( \mathfrak{osp}(1|2) \big)$ (and its extension $ \overline{\mathcal{U}} \big( \mathfrak{osp}(1|2) \big)$ which also allows infinite power series). 
The center of $ \mathcal{U} \big( \mathfrak{osp}(1|2) \big)$ is finitely generated by the Casimir element \cite{MR1773773}:
\begin{equation*}
C = \frac14 +\frac12\, b^-b^+ - \frac12\, b^+b^- + h^2 + 2 ef + 2fe. 
\end{equation*}
\begin{remark}
	The Casimir element $C$ differs from the Casimir element $\Om$ of $\mathfrak{sl}(2)$, given by \eqref{Om}, 
	by an additional term. 
	This extra term is related to another special element which we call the Scasimir element (see  \cite{MR1773773}).
	The Scasimir element 
	\[
	S  = \frac12\, b^-b^+ - \frac12\, b^+b^- -\frac12 
	\]
	is a square root of the Casimir operator: $S^2=C$. It 
	commutes with the even  (``bosonic'') generators and anti-commutes with the odd (``fermionic'') generators $b^+$ and $b^-$. 
	The Scasimir $S$ in our operator realization of $\mathfrak{osp}(1|2)$ is related to the angular Dirac operator or Gamma operator in Clifford analysis as follows:
	\begin{equation*}
	S  = \frac{m -1}{2}- \Gamma_x,\quad\text{with}\quad\Gamma_{x} =  - \ux \upx - \mE_x   = - \sum_{j<k} e_{jk} (x_{j} \partial_{x_{k}} - x_{k}\partial_{x_{j}}).
	\end{equation*}
\end{remark}
The Casimir element $C$ is a diagonal operator on the representation space $\mathrm{span}\{\, \psi_{j,k,\ell} \mid j \in \mZ_{\geq0}  \, \}$. 
Its action is given by
\begin{equation}\label{evC}
C \, \psi_{j,k,\ell} = \Big(k+\frac{m-1}{2}\Big)^2 \, \psi_{j,k,\ell}.
\end{equation}
Note that the eigenvalues of the Casimir operator $C$ are again squares of integers or half-integers depending on the value of the dimension $m$. 
However, contrary to the eigenvalues of the Casimir element $\Om$ in the harmonic case \eqref{evO}, the eigenvalues of $C$ are now squares of half-integers for even dimension, while for odd dimension we have squares of integers. 

The desired operators follow by using again the integer-valued polynomials defined in \eqref{Enx} and \eqref{Dnx}. 
We summarize this in the following theorem, which should be compared to \hyperref[sols]{Theorem~\ref*{sols}}.

\begin{theorem}\label{solsC}
	
	The properties
	\begin{itemize}
		\item[(I)] the Clifford-Helmholtz relations
		\begin{flalign*}
		&T \circ \partial_{\ux} = - i\, \uy \circ T &&\\
		&T \circ \ux =  - i \, \partial_{\uy} \circ T&&
		\end{flalign*}
		\item[(II)] $ T \, \psi_{j,k,\ell} = \mu_{j,k} \, \psi_{j,k,\ell} \qquad \text{with }\mu_{j,k}\in \mC$
		
		\item[(III)] $T^4 = \operatorname{id}$
	\end{itemize}
	are satisfied by operators $T$ of the form  
	\begin{equation}\label{TeC}
	T = \e^{i\frac{\pi}{2} F(\sqrt{C})
	}\e^{i \frac{\pi}{2}  ( h-\frac{m}{2} )   }  \in \overline{ \mathcal{U} } \big(  \mathfrak{osp}(1|2) \big)
	\end{equation}
	where $F(\sqrt{C})$ is an operator that consists of a function given by \eqref{anEn} (for odd dimension) or \eqref{anDn} (for even dimension)
	with the Casimir operator $C$ substituted for $x^2$. Moreover, every operator that satisfies properties {(I)}--{(III)} is equivalent with an operator of the form \eqref{TeO}.
\end{theorem}
\begin{proof}
	The last part follows by exactly the same reasoning as used in the proof of \hyperref[sols]{Theorem~\ref*{sols}}.
\end{proof}

In line with our approach in the previous section we now proceed by imposing a periodicity restriction to further narrow down this set of operators $T$.
This will again aid us in finding closed formulas for their kernels when written as integral transforms.

\subsection{Periodicity restriction}
\label{ssPrC}

The behavior of the eigenvalues $\mu_{j,k}$ of $T$ with regard to the index $j$ 
is given in \hyperref[eigenvC]{Lemma~\ref*{eigenvC}}.
By successive application of \eqref{mujkC} we find that these eigenvalues are four-periodic in $j$:
\[
\mu_{j+4,k} = (i)^4\, \mu_{j,k}=  \mu_{j,k}.
\]
For the same reasons as in the harmonic case we can again impose that the eigenvalues of $T$ should be 4-periodic in $k$. 
Depending on the parity of the dimension $m$ one works in, the desired operators follow from the results in either \hyperref[moddb]{Theorem~\ref*{moddb}} or \hyperref[moddb]{Theorem~\ref*{moddc}}, which of course remain valid. 
We summarize this in the following theorem.
\begin{theorem}\label{theo4perC}
	Let $
	T 
	$ be an operator that satisfies the following properties:
	\begin{itemize}
		\item[(I)] the Helmholtz relations
		\begin{flalign*}
		&T \circ  \Delta_{x}   \ =-  \lvert y \rvert^2\circ T &&\\
		&T \circ   \lvert x\rvert^2\,  = -  \Delta_{y} \circ T&&
		\end{flalign*}
		\item[(II)] $ T \, \phi_{j,k,\ell} = \mu_{j,k} \, \phi_{j,k,\ell} \qquad \text{with }\mu_{j,k}\in \mC$
		
		\item[(III)] $T^4 = \operatorname{id}$
		\item[(IV)] the eigenvalues of $T$ are 4-periodic in the index $k$: $ \mu_{j,k+4} =  \mu_{j,k}$
	\end{itemize}
	Then $T$ can be written as
	\begin{equation*}
	T = \e^{i\frac{\pi}{2} F(\sqrt{C})
	}\e^{i \frac{\pi}{2}  ( h-\frac{m}{2} )   }
	\end{equation*}
	where $F$ consists of a function of type \eqref{FE} as specified in \hyperref[moddb]{Theorem~\ref*{moddb}} (for odd dimension) or of type \eqref{FD} as in \hyperref[moddb]{Theorem~\ref*{moddc}} (for even dimension).
	
	Conversely, every operator $T$ of this form satisfies properties $(I)-(IV)$.
\end{theorem}
\begin{proof}
	Completely analogous to the proof of \hyperref[theo4per]{Theorem~\ref*{theo4per}}.
\end{proof}

\subsection{Closed formulas for the kernel}
\label{CftkC}

We are again interested in finding closed formulas for the kernel of the integral transforms corresponding to the determined operator solutions. 
It turns out that in even dimensions we have a result which can be compared to \hyperref[theoTabc]{Theorem~\ref*{theoTabc}} in the harmonic case. 
Hereto we first rewrite the kernel \eqref{kernelC} of 
\[
(T f)  (y) = \frac{1}{(2\pi)^{m/2}} \int_{\mathbb{R}^m}\! K_m(x,y) \, f(x) \, \mathrm{ d}x
\]
as obtained in \hyperref[intrC]{Proposition~\ref*{intrC}}. 
We put 
\begin{equation*}
K_m(x,y) = A_m(w,z) - \lambda \,  B_{m}(w,z) + (\ux \wedge \uy) \   z^{-1} \partial_w  B_m(w,z)
\end{equation*}
where
\begin{align}
\begin{split}
\label{AmBm}
A_m(w,z) &=  {2}^{\lambda}\, \Gamma ( \lambda )    \sum_{k=0}^{+\infty} (k+\lambda) \frac{1}{2} ( i \,  \mu_{0,k-1}+ \mu_{0,k})\, z^{-\lambda}J_{k+\lambda}( z) C^{\lambda}_{k}(w)\\
B_m(w,z) &=  2^{\lambda}\, \Gamma ( \lambda  )    \sum_{k=0}^{+\infty}  \frac{ 1}{2}\big(i \, \mu_{0,k-1}- \mu_{0,k}\big) \, z^{-\lambda}J_{k+\lambda}( z) C^{\lambda}_{k}(w) \\
z^{-1} \partial_w B_m(w,z) &=  2^{\lambda+1}\, \Gamma ( \lambda +1 )    \sum_{k=1}^{+\infty}  \frac{ 1}{2}\big(i \, \mu_{0,k-1}- \mu_{0,k}\big) \, z^{-\lambda-1}J_{k+\lambda}( z) C^{\lambda+1}_{k-1}(w),
\end{split}
\end{align}
using the notations 
$\lambda=(m-2)/2$, $z=|x |  |y|$ and $w = \langle x,y \rangle/z$. 
We then find the recursive relations listed in the following lemma.

\begin{lemma}\label{lemmakC}
	Let $T$ be an operator of the form \eqref{TeC} as specified in \hyperref[solsC]{Theorem~\ref*{solsC}}.  When written as an integral transform, the components of its kernel as specified in \eqref{AmBm} satisfy the following recursive relations
	\begin{equation*}
	A_{m+2} =  -i \, z^{-1} \partial_w A_m
	\end{equation*}
	\begin{equation*}
	B_{m+2} =   -i \, z^{-1} \partial_w B_m 
	\end{equation*}
	for $m\geq 2$. 
\end{lemma}
\begin{proof}
	As $T$ is of the form \eqref{TeC}, its eigenvalues are given by
	\[
	T  \, \psi_{j,k,\ell}  = \e^{i\frac{\pi}{2} F(\sqrt{C})
	}\e^{i \frac{\pi}{2}  ( h-\frac{m}{2} )   } \, \psi_{j,k,\ell} = \e^{i\frac{\pi}{2} F(k+\lambda+\frac12) }  \e^{i\frac{\pi}{2}(j+ k) }\, \psi_{j,k,\ell} =\mu_{j,k}\, \psi_{j,k,\ell}\rlap{,}
	\]
	where we used \eqref{evC} for the eigenvalues of $C$.
	
	Using property \eqref{Geg0} of the Gegenbauer polynomials, we find for $- i\, z^{-1} \partial_w A_m(w,z)$ the expression
	\begin{align*}
	&  {2}^{\lambda}\, \Gamma ( \lambda )    \sum_{k=0}^{+\infty} (k+\lambda) \frac{1}{2} ( i\, i^{ F(k-1+\lambda+\frac12) }  i^{k-1} +  i^{ F(k+\lambda+\frac12) }  i^k )\ z^{-\lambda-1}J_{k+\lambda}( z) \, \partial_w  C^{\lambda}_{k}(w) \\
	=\ &   {2}^{\lambda}\, \Gamma ( \lambda )    \sum_{k=1}^{+\infty} (k+\lambda)\, \frac{1}{2} (  i^{ F(k-1+\lambda+\frac12) }  i^{k} +  i^{ F(k+\lambda+\frac12) }  i^k ) \, z^{-(\lambda+1)}J_{k+\lambda}( z) \,2 \lambda \, C^{\lambda+1}_{k-1}(w)\\
	=\ &  {2}^{\lambda+1}\, \Gamma ( \lambda +1)    \sum_{k=0}^{+\infty} (k+1+\lambda)\, \frac{1}{2} (  i^{ F(k+\lambda+\frac12) }  i^{k} +  i^{ F(k+1+\lambda+\frac12) }  i^k ) \, z^{-(\lambda+1)}J_{k+1+\lambda}( z) \, C^{\lambda+1}_{k}(w).
	\end{align*}
	This is precisely $A_{m+2}(w,z)$. The second relation 
	follows in exactly the same way. 
\end{proof}

To obtain closed formulas for the kernels in even dimension, we again start by looking at dimension $m=2$ 
where we want to use the formulas
\eqref{cost}, \eqref{coss}, \eqref{sins}. 
An additional complication is that when written as an integral transform the kernel contains a non-scalar part. In dimension $m=2$ the Gegenbauer polynomials occurring in this bivector part reduce to sine functions. 
With this in mind we aim to use the formula 
\begin{equation}\label{sint}
\sin (  z \sin \theta) =  2 \sum_{n=0}^{\infty}  J_{2j+1}( z) \, \sin \bigl( (2n+1)\theta \bigr),
\end{equation}
which can be found in \cite[p.\,22, formula (2)]{MR0010746}. 
In $m=2$, the component $z^{-1} \partial_w B_m$ of \eqref{AmBm} 
will take on the form \eqref{sint} 
only if the eigenvalues 
of the corresponding operator $T$ satisfy, for $k\geq1$,
\[
\begin{cases}
i \, \mu_{0,k-1}- \mu_{0,k} = 0 &\text{if }k \text{ is even}\\
i \, \mu_{0,k-1}- \mu_{0,k} = c_0 &\text{if }k\text{ is odd},
\end{cases}
\]
with $c_0$ a constant independent of $k$. 
The operators of the form \eqref{TeC} whose eigenvalues satisfy these requirements are the subject of the following theorem.
In what follows we use the notations $s = \langle x,y\rangle$ and $t = \sqrt{\lvert x \rvert^2 \lvert y \rvert^2 - s^2}$. 
\begin{theorem}\label{theoTab2}
	In dimension $m=2$, the operator exponential
	\begin{equation*}
	T_{ab}=   \e^{i\frac{\pi}{2}  F_{ab}(\sqrt{ C})}\e^{i \frac{\pi}{2}  ( h-1 )   }
	\end{equation*}
	with $F_{ab}(x) = a 
	+ b \,\cD_{0110} (x) $ (as specified in \hyperref[moddc]{Theorem~\ref*{moddc}}) and $a,b\in \mZ_4$, 
	can be written as an integral transform whose kernel is given by
	\begin{equation}\label{kernelab2}
	K_2(x,y) = i^a\Bigl(  \frac{1+ i^b}{2} \bigl(  \cos(s) + i  \sin(s)\bigr)  + \frac{1- i^b}{2} \cos(t)+ (\ux \wedge \uy) \   \frac{1- i^b}{2} \, \frac{i\sin ( t)}{t} \Bigr).
	\end{equation}
\end{theorem}
\begin{proof}
	The proof goes along the same lines as the proof of \hyperref[theoTabc2]{Theorem~\ref*{theoTabc2}}.
	
	By \hyperref[solsC]{Theorem~\ref*{solsC}}, $T_{ab}$ satisfies the properties {(I)}--{(III)}. 
	For $m=2$, the kernel obtained in \hyperref[intrC]{Proposition~\ref*{intrC}} (for the formulation of
	$T$ as an integral transform), reduces to
	\begin{equation*}
	K_2(x,y) = \mu_{0,0} \, J_0(z) +    \sum_{k=1}^{+\infty}  ( i \,  \mu_{0,k-1}+ \mu_{0,k}) \, J_{k}( z) \, \cos (k \theta)  +   (\ux \wedge \uy) \    \sum_{k=1}^{+\infty}  ( i \,  \mu_{0,k-1}- \mu_{0,k}) \, J_{k}( z) \,  \frac{\sin (k\theta)}{ z \sin\theta} .
	\end{equation*} 
	This follows from $\lambda= 0$ and the identities (see \cite[Vol. {I}, section 3.15]{Erde}, formulas (14) and (15)), for $w = \cos \theta$ and integer $k\geq1$,
	\[
	\lim_{\lambda \to 0} \Gamma ( \lambda ) \ C_k^\lambda ( \cos \theta) = \frac{2}{ k} \cos (k \theta), 
	\quad \text{ and } \quad
	C_{k-1}^1(\cos \theta) = \frac{\sin (k\theta)}{\sin\theta}\rlap{.}
	\]
	Now, using the fact that the eigenvalues of $T_{ab}$ are 4-periodic in $k$, with (for $m=2$)
	\[
	\mu_{0,0} = i^a,\quad  \mu_{0,1} = i^{a+b} i,\quad  \mu_{0,2} =i^{a +b} (-1) ,\quad  \mu_{0,3} = i^{a} (-i),
	\]
	we can rewrite the kernel as
	\begin{align*}
	K_2(x,y) =  i^a \biggl( & J_0(z) +   2 \sum_{n=1}^{+\infty}J_{4n}( z) \, \cos (4n \theta) - i^{b}\,   2 \sum_{n=0}^{+\infty} J_{4n+2}( z) \, \cos \bigl((4n+2) \theta\bigr)  \\
	&+ (1+  i^{b})\,  i\,\Bigl(   \sum_{n=0}^{+\infty} J_{4n+1}( z) \, \cos  \bigl((4n+1) \theta\bigr)
	-   \sum_{n=0}^{+\infty}  J_{4n+3}( z) \, \cos \bigl((4n+3) \theta\bigr) \Bigr) \\
	& +  (\ux \wedge \uy) \    (1-  i^{b})\,i\,   \sum_{n=0}^{+\infty} J_{2n+1}( z)  \, \frac{\sin  \bigl((2n+1) \theta\bigr)}{ z \sin \theta}   \biggr) .
	\end{align*} 
	The formulas \eqref{cost}, \eqref{coss}, \eqref{sins}, \eqref{sint} and $s=zw=z\cos\theta$, $t=z \sqrt{1-w^2} = z \sin \theta$ then yield
	\begin{equation*}
	K_2(x,y) = i^a\Bigl( \frac{1}{2} \bigl(\cos(s) + \cos(t) \bigr) + \frac{ i^b}{2} \bigl(\cos(s) - \cos(t) \bigr) +  \frac{(1+  i^{b})i}{2} \sin(s) \Bigr)+ (\ux \wedge \uy) \   \frac{1- i^b}{2} \, \frac{i\sin ( t)}{t} \Bigr).
	\end{equation*}
\end{proof}

Using \hyperref[lemmakC]{Lemma~\ref*{lemmakC}} we now find:

\begin{theorem}\label{theoTab}
	In even dimension $m$, the operator exponential 
	\begin{equation*}
	T_{ab}=   \e^{i\frac{\pi}{2}  F_{ab}(\sqrt{C})}\e^{i \frac{\pi}{2}  ( h-\frac{m}{2} )   }
	\end{equation*}
	with $F_{ab}(x) = a 
	+ b \,\cD_{0110} (x) $ (as specified in \hyperref[moddc]{Theorem~\ref*{moddc}}), 
	can be written as an integral transform whose kernel is given by
	\begin{align}
	\begin{split}
	\label{kernelab}
	K_m(x,y) = \ & i^{a-\lambda} \, \biggl(  \frac{1- i^b}{2} \Bigl( \partial_s -  \frac{s}{t} \   \partial_t \Bigr)^{\lambda} \cos(t) - \lambda    \frac{1- i^b}{2} \, \Bigl( \partial_s -  \frac{s}{t} \   \partial_t \Bigr)^{\lambda-1} \frac{i \sin ( t)}{t}\\ & 
	\ +\frac{1+ i^b}{2} \bigl(  \cos(s) + i\,  \sin(s)\bigr)   + (\ux \wedge \uy) \   \frac{1- i^b}{2} \Bigl( \partial_s -  \frac{s}{t} \   \partial_t \Bigr)^{\lambda} \frac{i \sin ( t)}{t}  \biggr),
	\end{split}
	\end{align}
	with $\lambda=(m-2)/2$. Moreover, one has
	\begin{align*}
	\Bigl( \partial_s -  \frac{s}{t} \   \partial_t \Bigr)^{\lambda}\cos(t) =  \sqrt{\frac{\pi}{2}} \sum_{\ell =0}^{\left\lfloor  \frac{m-2}{4} \right\rfloor} s^{\frac{m}{2}-1-2 \ell}\    \frac{1}{2^{\ell} \ell!} \frac{\Gamma(\frac{m}{2})}{\Gamma(\frac{m}{2}-2\ell)}   \frac{ {J}_{(m-2\ell-3)/2}(t)}{t^{(m-2\ell-3)/2}} .
	\\
	\Bigl( \partial_s -  \frac{s}{t} \   \partial_t \Bigr)^{\lambda-1} \frac{\sin ( t)}{t} =  \sqrt{\frac{\pi}{2}} \sum_{\ell =0}^{\left\lfloor  \frac{m-4}{4} \right\rfloor} s^{\frac{m}{2}-2-2 \ell}\    \frac{1}{2^{\ell} \ell!} \frac{\Gamma(\frac{m}{2}-1)}{\Gamma(\frac{m}{2}-1-2\ell)}   \frac{ {J}_{(m-2\ell-3)/2}(t)}{t^{(m-2\ell-3)/2}} .
	\\
	\Bigl( \partial_s -  \frac{s}{t} \   \partial_t \Bigr)^{\lambda} \frac{\sin ( t)}{t} =  \sqrt{\frac{\pi}{2}} \sum_{\ell =0}^{\left\lfloor  \frac{m-2}{4} \right\rfloor} s^{\frac{m}{2}-1-2 \ell}\    \frac{1}{2^{\ell} \ell!} \frac{\Gamma(\frac{m}{2})}{\Gamma(\frac{m}{2}-2\ell)}   \frac{ {J}_{(m-2\ell-1)/2}(t)}{t^{(m-2\ell-1)/2}} .
	\end{align*}
\end{theorem}

\begin{proof}
	This is analogous to the proof of \hyperref[theoTabc]{Theorem~\ref*{theoTabc}} in the harmonic case. 
	For $m=2$ we have $\lambda = 0$, and the expression \eqref{kernelab} coincides with
	\eqref{kernelab2} which was obtained in 
	\hyperref[theoTab2]{Theorem~\ref*{theoTab2}}. 
	Starting from $K_2$, we find the kernel in every even dimension $m>2$ 
	by means of the recursive relations obtained in \hyperref[lemmakC]{Lemma~\ref*{lemmakC}}. 
	Note that when moving from dimension $m=2$ to $m=4$, although we have no explicit expression for $B_2$, we can obtain $B_{4}$ directly from the bivector part in $K_2$:
	\begin{equation*}
	K_2(x,y) = A_2(w,z) + (\ux \wedge \uy) \   z^{-1} \partial_w  B_2(w,z)= A_2(w,z)  + (\ux \wedge \uy) \  i \, B_4(w,z).
	\end{equation*}
\end{proof}

\subsubsection{Specific cases}

The kernel \eqref{kernelab} in \hyperref[theoTab]{Theorem~\ref*{theoTab}} consists of a linear interpolation of two terms, with coefficients $(1+i^b)/2$ and $(1-i^b)/2$. 
The first term is precisely the kernel of the classical Fourier transform, while the second term contains a non-scalar part and can be written as a finite sum of Bessel functions. 
This second term closely resembles an expression found for the kernel of the Clifford-Fourier transform in
\cite{DBXu} and also for similar transforms devised in \cite{DBNS}.

We can isolate this second term by putting $b=2$ in \eqref{kernelab}. 
Multiplying by $ i^{\lambda-a}$ to eliminate a scalar multiplicative factor, the kernel reduces to
\begin{align*}
K_m(x,y) =  \Bigl(  \Bigl( \partial_s -  \frac{s}{t} \   \partial_t \Bigr)^{\lambda} \cos(t) - \lambda  \,i \Bigl( \partial_s -  \frac{s}{t} \   \partial_t \Bigr)^{\lambda-1} \frac{\sin ( t)}{t} \Bigr)   
+ (\ux \wedge \uy) \   i\,  \Bigl( \partial_s -  \frac{s}{t} \   \partial_t \Bigr)^{\lambda} \frac{\sin ( t)}{t},
\end{align*}
which corresponds to the operator exponential
\begin{equation*}
T = \e^{i\frac{\pi}{2}\lambda  } \e^{i\frac{\pi}{2}( C - \frac14 )
}\e^{i \frac{\pi}{2}  ( h-\frac{m}{2} )   }  
= \e^{i\frac{\pi}{2}( C - \frac54 + h )   } 
.
\end{equation*}
For dimension $m=2$ this kernel reduces to
\begin{equation*}
K_2(x,y) =  \cos(t)+ (\ux \wedge \uy) \ i \,  \frac{\sin ( t)}{t}.
\end{equation*}

\section{Conclusions and remarks}
\label{sec:4}

We started our investigation from a list of properties {(i)}--{(iii)} which 
specifically highlights two symmetric structures underlying the higher dimensional Fourier transform, namely the orthogonal symmetry and an algebraic symmetry with respect to an operator realization of the Lie algebra $\mathfrak{sl}(2)$. 
We then obtained a complete set of solutions satisfying {(i)}--{(iii)} in the form of operator exponentials, which includes in particular the classical Fourier transform. 
The transforms we have constructed also demonstrate other interesting features. 
For instance, for each of them, we found a corresponding formulation as an integral transform. Moreover, for a select set of operators when written as an integral transform the kernel could even be reduced to a closed formula being polynomially bounded. 

In the process of describing these solutions, we gave a brief overview on the subject of integer-valued polynomials and a generalization thereof, i.e.~polynomials that are integer-valued on the set of square numbers, the set of half-integers or the set of squares of half-integers. 

These results were subsequently lifted to the setting of Clifford analysis where
the Lie algebra $\mathfrak{sl}(2)$ in 
the algebraic symmetry was generalized to the Lie superalgebra $\mathfrak{osp}(1|2)$, resulting in a new list of properties {(I)}--{(III)}. 
Also in this case the obtained operators form a complete set of solutions and they all have an equivalent formulation as an integral transform. 
For a select set of transforms we again find a polynomially bounded formula for the kernel.
The findings obtained here coalesce with those for similar generalized Fourier transforms in the context of Clifford analysis as described in 
\cite{DBNS,DBXu}.

We note that in both the regular and the Clifford setting more closed formulas for kernels are found for even dimension than in the odd dimensional case. 
The reason for this is that, although the recursive relations obtained in \hyperref[lemmakC]{Lemma~\ref*{lemmak}} and \hyperref[lemmakC]{Lemma~\ref*{lemmakC}} also hold for odd dimension, we have no formula for dimension $m=3$ to use as starting point and as a consequence no such form is found.

Finally, we note that when moving to Clifford analysis we replaced the Helmholtz relations by their more restrictive variant called the Clifford-Helmholtz relations. 
Nevertheless, one may also consider solutions to the regular Helmholtz relations in the context of Clifford analysis. 
Following the same strategy we applied before, this leads to the subset of $ \mathcal{U} \big( \mathfrak{osp}(1|2) \big)$ consisting of the elements that commute with all elements of $\mathfrak{sl}(2)$, instead of $\mathfrak{osp}(1|2)$. This subset is generated by $S$, the Scasimir element of $\mathfrak{osp}(1|2)$, whose eigenvalues on the basis $\{\psi_{j,k,\ell}\}$ square to those of the Casimir $C$, given by \eqref{evC}. Accordingly, this brings us to polynomials in $S$ that again take integer values when acting on the basis $\{\psi_{j,k,\ell}\}$. This gives rise to a bigger class of more general operator exponentials which includes in particular the operators we obtained by imposing the Clifford-Helmholtz relations.

\appendix

\section{Appendix}
\label{sec:5}

Here we give an overview of some definitions and results used in the main text, as well as some proofs that have been omitted from the text.

\subsection{Special functions}
\label{secA1}

For $\alpha>-1$, $p$ a positive integer and $\Gamma(\cdot)$ the Gamma function, the generalized Laguerre polynomials are given by
\begin{equation*}
L_j^{\alpha} (x) = \sum_{n=0}^j \frac{\Gamma(j+\alpha+1)}{n! \, (j-n)! \, \Gamma(n+\alpha+1)} (-x)^n.
\end{equation*}
For $k\in \mZ_{\geq0}$ and $\lambda > -1/2$, the Gegenbauer polynomials are defined as
\begin{equation*}
C_k^{\lambda}(w) = \sum_{j=0}^{\lfloor \frac{k}{2} \rfloor} (-1)^j \frac{\Gamma ( k-j+\lambda)}{\Gamma(\lambda)\, j!\, (k-2j)!} (2w)^{n-2j}.
\end{equation*}
They are a special case of the Jacobi polynomials and satisfy 
the differentiation 
property (see e.g.~\cite{MR1827871,Sz}):
\begin{equation}\label{Geg0}
\frac{\mathrm{d}}{\mathrm{d}w} 
C_k^{\lambda}(w) 
= 2 \lambda \ C_{k-1}^{\lambda + 1}(w).
\end{equation}

The Bessel function can be defined by the power series
\begin{equation}\label{taylorbessel}
J_\nu(t) = \sum_{n=0}^{\infty} \frac{(-1)^n}{n!\ \Gamma (n+\nu+1)} \left ( \frac{t}{2} \right )^{2n+\nu},
\end{equation}
where $\nu\in\mathbb{C}$ is the order of the Bessel function (see e.g.~\cite{MR0010746}). 
In particular, for $\nu={1}/{2}$ and $\nu=-{1}/{2}$, the power series \eqref{taylorbessel} reduces to 
\begin{equation}
\label{sincos1}
J_{1/2} (t)  = \sqrt{\frac{2}{\pi t } } \sin t,\qquad
J_{-1/2} (t)  = \sqrt{\frac{2}{\pi t } } \cos t,
\end{equation}
as found in \cite[p.\,54]{MR0010746}.
The Bessel functions satisfy (see \cite[p.\,45]{MR0010746})
\begin{equation}\label{ddt}
-\frac{1}{t} \frac{\mathrm d}{\mathrm{d}  t} 
\lbrack t^{-\nu} J_\nu( t) \bigr \rbrack=  t^{-(\nu+1)}J_{\nu+1}( t).
\end{equation}

\subsection{Integer-valued polynomials}
\label{sec:A2}

We will now elaborate on the identities showcased 
in \hyperref[moddb]{Theorem~\ref*{moddb}} and \hyperref[moddc]{Theorem~\ref*{moddc}}
of which the proofs were omitted from the main text. 
Before delving into the actual proofs, we first show some auxiliary results. 
In subsection~\ref{ivp} we already mentioned the polynomials $E_n(x)$ and $D_n(x)$ exhibiting a property of periodicity when evaluated modulo 4. 
We now prove a similar result for the more general integer-valued polynomials $\binom{x}{n}$ from which the other periodicity results follow. 
\begin{proposition}\label{period}
	For $n,N\in \mZ_{\geq0}$ such that $n < 2^N$, and $x\in \mZ$, one has
	\[
	\binom{x}{n} \equiv \binom{x+2^{N+1}}{n} \pmod{ 4}
	\]
\end{proposition}
\begin{proof}
	Take $N\in \mZ_{\geq0}$ and $n=2^N-1$, the proof for every $n<2^N-1$ is similar. 
	
	From \eqref{xchn}, we have by definition
	\begin{equation}\label{x2N}
	\binom{x+2^{N+1}}{2^N-1} = \frac{1}{(2^N-1)!} \prod_{k=0}^{2^N-2} (2^{N+1}+x-k).
	\end{equation}
	When expanding this product, each term consists of a number of factors of the form $x-k$ multiplied with a power of $2^{N+1}$. We will show that every term containing a factor $2^{N+1}$ is divisible by four and hence congruent 0 modulo 4.
	
	First, notice that the denominator $(2^N-1)!$ in \eqref{x2N} contains $2^N-N-1$ times the factor 2. This follows from adding the number of integers smaller than $2^N-1$ that are divisible by the powers $2^j$ for $j$ ranging from 1 to $N-1$. 
	
	Next, we determine when a term containing a factor $2^{N+1}$ in the expansion of the numerator of \eqref{x2N} has a minimal number of factors 2 in its prime factorization. 
	This is the case when its other factors in the product \eqref{x2N} are the integers of the set $A_x \defeq\{\, x-k \mid k\in \mZ_{\geq0} \text{ and } 0\leq k \leq 2^{N}-2 \,\}$ 
	that contain less than $N+1$ factors 2 and its remaining factors being $2^{N+1}$. For a given $x\in\mZ$, the set $A_x$ consists of $2^N-1$ consecutive integers. 
	The absolute minimal case thus ensues when the set $A_x$ contains $2^{N}-2$ integers that are not divisible by $2^{N-1}$.  
	The product of these integers gives a minimum of $2^{N}-2N$ factors 2. 
	Hence, a factor $2^{N+1}$ multiplied by such a product of integers contains at least $2^N-N+1$ times the factor 2.
	
	As the denominator in \eqref{x2N} contains $2^N-N-1$ times the factor 2, 
	every term in the expansion of \eqref{x2N} with a factor $2^{N+1}$ must contain at least two factors 2 or thus be divisible by four.
\end{proof}

\begin{corollary}\label{periodE}
	For natural numbers $n,N\in \mZ_{\geq0}$ such that $n \leq 2^N$, and $x\in \mZ_{\geq0}$, one has
	\[
	E_{n}(x) \equiv E_n\big(x+2^{N+2}\big) \pmod{ 4},
	\]
	and
	\[
	2\, E_{n}(x) \equiv 2\, E_n\big(x+2^{N+1}\big) \pmod{ 4},
	\]
\end{corollary}
\begin{proof}
	Writing \eqref{Enx} as
	\begin{equation*}
	E_n(x) 
	= \frac{2}{(2n)!} \prod_{k=0}^{n-1} (x^2-k^2) =  \frac{2}{(2n)!} \ x \prod_{k=0}^{2n-2} (x-n+1+k) ,
	\end{equation*}
	the periodicity follows by the same reasoning as used in the proof of \hyperref[period]{Proposition~\ref*{period}}.
\end{proof}

We need one more lemma which will prove to be useful in the proof of \hyperref[moddb]{Theorem~\ref*{moddb}}.

\begin{lemma}\label{oddE}
	For $n$ an odd integer and $k$ an even integer, we have
	\[
	E_{n}(k) \equiv 0 \pmod{4}.
	\]
\end{lemma}
\begin{proof}
	Put $n=2N+1$ and $k=2K$ for some integers $N,K$. Then
	\[
	E_{2N+1}(2K) =\frac{2K}{2N+1} \binom{2K+2N}{4N+1} =\frac{2K(2K+2N)}{(2N+1)(4N+1)} \binom{2K+2N-1}{4N}\equiv 0 \pmod{4}
	\]
	as the numerator contains a factor 4 while the denominator contains only odd factors.
\end{proof}

\begin{proof}[Proof of {\protect\hyperref[moddb]{Theorem~\ref*{moddb}}}]\label{proofb}
	Note first that for $a,b,c,d \in \{0,1,2,3\}$ the function
	\[
	F(x) = a 
	+ b \,(\mathbf{1}_{4\mZ+1}(x) + \mathbf{1}_{4\mZ+3}(x)) + c \,\mathbf{1}_{4\mZ+2}(x)+d\, \mathbf{1}_{4\mZ+3}(x)
	\]
	is 4-periodic with its first four values given by
	\[
	F(0)=a,\quad F(1)=a+b,\quad F(2)= a + c,\quad F(3)= a + b +d.
	\]
	By specifying $a,b,c,d\in \{0,1,2,3\}$ one can clearly obtain every possible combination of first four values modulo 4.
	
	The first identity of \hyperref[moddb]{Theorem~\ref*{moddb}} follows easily from the fact that
	$
	E_1(x) = x^2,
	$ 
	as every $x\in\mZ_{\geq0}$ is equal to either $2k$ or $2k+1$ for some $k\in \mZ_{\geq0}$, so
	\[
	(2k)^2 = 4k^2  \equiv  0  \pmod{4}, \quad\text{while}\quad
	(2k+1)^2 = 4k^2 + 4k +1 \equiv  1  \pmod{4} .
	\]
	
	For the second identity, by \eqref{Enx} we have 
	\[
	E_2(x) + 2\, E_3(x) =\frac{1}{12}x^2(x^2-1)+ \frac{1}{180}x^2(x^2-1)(x^2-4)=\frac{x^2(x^2-1)(x^2+11)}{4\cdot 45}.
	\]
	Now for $k \in \mZ_{\geq0}$, we find
	\[
	E_2(2k+1) + 2\, E_3(2k+1) =\frac{(4k^2+4k+1)(4k^2+4k)(4k^2+4k+12)}{4\cdot 45} \equiv 0 \pmod{ 4}.
	\]
	Moreover, again for $k \in \mZ_{\geq0}$, it holds that
	\[
	E_2(2k) + 2\, E_3(2k) =\frac{4k^2(4k^2-1)(4k^2+11)}{4\cdot 45} \equiv k^2 \pmod{ 4}.
	\]
	We already know that $k^2$ is congruent 1 modulo 4 when $k$ is odd, which corresponds to $2k \equiv 2 \pmod{4}$, and congruent 0 modulo 4 when $k$ is even, corresponding to $2k \equiv 0 \pmod{4}$.
	
	For the last identity, we will show that 
	\[
	\cE_{0001}(x) \equiv 
	\mathbf{1}_{4\mZ+3}(x)  \pmod{ 4} .
	\]
	Hereto we write
	\[
	\cE_{0001}(x) = \sum_{n=1}^{\infty}T_n(x) \qquad \text{with}\qquad T_n(x) \defeq  E_{2^n+1}(x) + \sum_{j=1}^{n-1} 2\, E_{2^n+1+2^j}(x).
	\]
	By \hyperref[oddE]{Lemma~\ref*{oddE}} we immediately have that for $x$ an even integer each term in $T_n(x)$ is congruent 0 modulo 4. 
	What remains to be shown is that for $k$ an integer one has $\cE_{0001}(4k+1) \equiv 0\pmod{ 4}$ and $\cE_{0001}(4k+3) \equiv 1 \pmod{ 4}$. 
	Hereto we look at the values of the functions $E_{2^n+1}$ and $ 2\, E_{2^n+1+2^j}$ that constitute $T_n$ and in this way generate the values of $T_n$ and ultimately those of $\cE_{0001}$.
	
	Recall that 
	\begin{equation*} 
	E_n(x) 
	=   \frac{\prod_{\ell=0}^{n-1} (x^2-\ell^2)}{\prod_{\ell=0}^{n-1} (n^2-\ell^2)}.
	\end{equation*}
	We are dealing with functions that consist of a fraction with integer numerator and denominator which evaluates to an integer. 
	To determine the value of this integer up to modulo 4 congruence, we will look at the numerator and the denominator separately and factorize them into a power of 2 and a remaining odd part. 
	For instance, if $
	{a}/{b}$ is such a fraction and we have the following factorization $a = 2^k a'$ and $b=2^{\ell}  b'$ with $a'$ and $b'$ odd integers, then 
	\[
	\frac{a}{b}\! \mod{4}\, \equiv 
	\begin{cases}
	\quad 0 & \text{if }k-\ell\geq 2\\
	\quad  2  &\text{if } k-\ell= 1\\
	{a'}\cdot ({b'})^{-1}  &\text{if } k=\ell.
	\end{cases} 
	\]
	Note that $k\geq\ell$ must hold as $2^{-1}$ is not defined modulo 4 and we know that $a/b$ must evaluate to an integer. Moreover, modulo 4 we have for odd integers $(1)^{-1}\equiv 1\pmod{4}$ and $(3)^{-1}\equiv 3\pmod{4}$. 
	
	We first look at the denominator of the function $E_{2^n+1}(x) $ for $n\geq 1$ which is given by 
	\begin{equation}
	\label{prod}
	\prod_{\ell=0}^{2^{n}} \bigl(   (2^n+1)^2   -  \ell^2\bigr) 
	\equiv  \prod_{j=0}^{2^{n-1}-1} \bigl(   (2^n+1)^2   -  (2j+1)^2\bigr)\pmod{4}.
	\end{equation}
	Here we 
	carried out a first simplification where we removed a number of factors congruent 1 modulo 4 as they have no impact on the final value modulo 4 of the product. 
	Indeed, for $\ell$ even, say $\ell=2j$, we have $(2^n+1)^2   -  (2j)^2 = 4n^2+4n+1-4j^2\equiv  1 \pmod{ 4} $. 
	If $n=1$, \eqref{prod} reduces to just one factor, yielding 
	\[
	\prod_{j=0}^{0} \bigl(   (2+1)^2   -  (2j+1)^2\bigr) =   \bigl(   3^2   - 1^2\bigr) =   2^3
	\]
	for the denominator of $E_3(x)$. For $n\geq 2$ we work out \eqref{prod} as follows
	\begin{align*}
	\prod_{j=0}^{2^{n-1}-1} \bigl(   (2^n+1)^2   -  (2j+1)^2\bigr) 
	& =  \prod_{j=0}^{2^{n-1}-1} \bigl(   2^n  +1- 2j-1\bigr) \bigl(   2^n  +1+  2j+1\bigr)  \\
	&   =  \prod_{j=0}^{2^{n-1}-1} \bigl(   2^n  - 2j\bigr) \bigl(   2^n  +  2j+2\bigr)  \\ &   =      (2^{n+1})!! \\ & = 2^{(2^n)} (2^n)!
	\end{align*}
	Using repeatedly the recursive relation (for $n\geq 1$)
	\[
	(2^n)! = 2^{(2^{n-1})} (2^n-1)!! \, (2^{n-1})!
	\]
	we find
	\begin{align*}
	2^{(2^n)} (2^n)! &= 2^{(2^n)} \prod_{j=1}^{n} 2^{(2^{j-1})} (2^j-1)!! \\
	& = 2^{( 2^{n+1} - 1)}  \prod_{j=1}^{n} (2^j-1)!! 
	\end{align*}
	Now, the double factorial $(2^j-1)!!$ consists of a product of $2^{j-1}$ odd consecutive integers. 
	For $j = 1$ this product consists of just one factor, namely 1. The product of two odd consecutive integers is always congruent 3 modulo 4, hence $j = 2$ contributes a factor 3, while for $\ell\geq 3$ this double factorial is thus necessarily congruent 1 modulo 4. 
	In this way, for the denominator of $E_{2^n+1}(x) $ we arrive at 
	\begin{equation}\label{denom}
	2^{ 2^{n+1} - 1} \prod_{\ell=0}^{n-1} \prod_{j=0}^{2^{\ell}-1} (2j+1)   \equiv  2^{ (2^{n+1} - 1)}   \cdot 3\pmod{4}   . 
	\end{equation}
	
	Using this information we will show that for integers $k$ and $x$
	\[
	\qquad E_{2^n+1}(2^nk+2x+1)\equiv \frac{k(k+1)}{2}   \pmod{ 4} \qquad (0\leq x < 2^{n-1}).
	\]
	For $n=1$ we immediately have 
	\[
	E_3(2k+1) \equiv  \frac{1}{2^3}\prod_{j=0}^{2} \bigl(   (2k+1)^2   -  j^2\bigr) \equiv \frac{4k^2+4k+1-1}{2^3}\equiv \frac{k(k+1)}{2}   \pmod{ 4},
	\]
	while for $n\geq2$ we find for the numerator of $E_{2^n+1}(2^nk+2x+1) $: 
	\begin{align*}
	\prod_{j=0}^{2^{n}} \bigl(   (2^nk+2x+1)^2   -  j^2\bigr) & \equiv 
	\prod_{j=0}^{2^{n-1}-1} \bigl(   (2^nk+2x+1)^2   -  (2j+1)^2\bigr)  \pmod{ 4}\\
	& = \prod_{j=0}^{2^{n-1}-1}\bigl(   2^nk +2x - 2j\bigr)  \bigl(   2^nk +2x +  2j+2\bigr)   \\
	& = \prod_{j=1}^{2^{n}}\bigl(   2^n(k-1) +2x  +  2j\bigr) .
	\end{align*}
	Here we have a product of $2^n$ even consecutive integers starting at $ 2^n(k-1) +2x+2$. For $x=0$ this product starts at $ 2^n(k-1) +2$ and goes up to $2^n(k-1) +2^{n+1}$. 
	Compared to the case $x=0$, a shift occurs when $x> 0$ where for $j$ from 1 up to $x$ 
	the factor
	$
	2^n(k-1) + 2j 
	$ in this product gets replaced by 
	$
	2^n(k-1) + 2j +2^{n+1}
	$. 
	Now, our goal is to factor out all powers of 2 and look at the remaining odd part modulo 4. 
	As long as $2j<2^{n}$, or thus $x<2^{n-1}$, we see that $2^n(k-1) + 2j $ contains at most a power of 2 equal to $2^{n-1}$. 
	Hence, when factoring out all powers of 2 the added term $2^{n+1}$ 
	in the replacement factor $2^n(k-1) + 2j +2^{n+1}$ can at most reduce to $2^2$. This means that the odd part remaining after factoring out all powers of 2 of 
	the replacement factor and the original factor give the same value modulo 4. 
	Note that this no longer holds for $x=2^{n-1}$ as one can then factor out $2^n$.

	For $0\leq x < 2^{n-1}$ we have
	\[ 
	\prod_{j=1}^{2^{n}}(   2^n(k  -1)   +2x  +  2j\bigr)   
	\equiv \prod_{j=1}^{2^{n}}(   2^n(k  -1)     +  2j\bigr)   \pmod{ 4} 
	\]
	Using repeatedly the recursive relation (for $n\geq 1$)
	\[
	\prod_{j=1}^{2^{n}}(   2^n(k  -1)     +  2j\bigr)   
	= 2^{(2^n)} \prod_{j=1}^{2^{n-1}}(   2^{n-1}(k  -1)     +  2j-1\bigr)  \prod_{j=1}^{2^{n-1}}(   2^{n-1}(k  -1)     +  2j\bigr)   
	\]
	we find
	\begin{align*} 
	\prod_{j=1}^{2^{n}}(   2^n(k  -1)     +  2j\bigr) 
	& =(k+1)
	\prod_{\ell=1}^{n}  2^{(2^{\ell})} \prod_{j=1}^{2^{\ell-1}} \bigl( 2^{\ell-1}(k-1)+ 2j-1\bigr)   \\
	& =(k+1) 2^{( 2^{n+1}-2)} \prod_{\ell=1}^{n}  \prod_{j=1}^{2^{\ell-1}} \bigl( 2^{\ell-1}(k-1)+ 2j-1\bigr)  .
	\end{align*}
	Now, for $\ell=1$ the inner most product reduces to one factor $2^0(k-1)+1=k$, while for $\ell=2$ we have $\bigl( 2(k-1)+1\bigr)  \bigl( 2(k-1)+ 3\bigr) \equiv 3 \pmod{4}$. 
	For all other values of $\ell$ we have a product of $2^{\ell-1}$ odd consecutive integers which is always congruent 1 modulo 4. 
	In this way we arrive at
	\[
	2^{ (2^{n+1} - 2)}   \cdot 3 \cdot k (k+1)
	\]
	for the numerator of $E_{2^n+1}(2^nk+2x+1) $.
	Together with what we obtained for the denominator of $E_{2^n+1}(x) $ in \eqref{denom} we ultimately find
	\[
	\qquad E_{2^n+1}(2^nk+2x+1)\equiv \frac{k(k+1)}{2}   \pmod{ 4} \qquad \bigl(0\leq x < 2^{n-1}\bigr) .
	\]
	
	What this means is that when looking at the values of $E_{2^n+1}$ modulo 4 evaluated at the odd integers, starting at 1 we have  $2^{n-1}$ times the value 0, followed by $2^{n-1}$ times the value 1, then $2^{n-1}$ times the value 3 and $2^{n-1}$ times the value 2 after which the sequence mirrors and repeats due to the periodicity result in \hyperref[periodE]{Corollary~\ref*{periodE}} and the fact that $E_n(-x)=E_n(x)$. 
	This is illustrated in Table~\ref{tableE2n1}.
	
	Using the same techniques as above one shows that
	\[
	2E_{2^n+1+2^j}\bigl(2^n+2^{n+1}k+2x+2^jy+1\bigr)\equiv {(k+1)^2}y(y+1)   \pmod{ 4},
	\]
	for $0\leq x <2^{j-1}$ and $  0\leq y < 2^{n-j+1}$. 
	Note that because of the added factor 2 in front, it suffices to count the powers of 2 and we do not need to take into account the congruence class modulo 4 of the remaining odd part when decomposing the numerator and the denominator. 
	This result translates to the evaluation at odd integers as follows:  starting at 1, for $ 2E_{2^n+1+2^j}$ we have $2^{n-1}$ times the value 0, followed by 
	$2^{n-j+1}$ times a sequence consisting of $2^{j-1}$ values 0, $2^{j}$ values 2 and again $2^{j-1}$ values 0. This is followed by $2^{n-1}$ times the value 0 and the other values follow from \hyperref[periodE]{Corollary~\ref*{periodE}}.
	
	Together with what we obtained for $E_{2^n+1}$, this gives the following sequence of values for $T_{n}$ with $n\geq 2$ evaluated at the integers $4k+1$ starting at 1 ($k=0$): we have $2^{n-2}$ times the value 0, followed by $2^{n-1}$ times the value 1, $2^{n-1}$ times the value 2, $2^{n-1}$ times the value 3 and finally $2^{n-2}$ times the value 0, after which the sequence repeats due to the periodicity result in \hyperref[periodE]{Corollary~\ref*{periodE}}. 
	This is illustrated in Table~\ref{tableTn1}. A similar result holds for the values of $T_{n}$ evaluated at the integers $4k+3$ which is illustrated in table~\ref{tableTn3}. 
	More formally, 
	we can also write this as (for $n\geq 2$) 
	\begin{equation}
	\begin{split}\label{Tn413}
	T_n(2^nk+4x+1) &\equiv  \frac{k(k+1)}{2} + \frac{k^2(k^2-1)}{2}  \pmod{ 4} \qquad (0\leq x < 2^{n-2})\\
	T_n(2^nk+4x+3) &\equiv \frac{k(k-1)}{2} + \frac{k^2(k^2-1)}{2}  \pmod{ 4} \qquad (0\leq x < 2^{n-2}).
	\end{split}
	\end{equation}
	
	Putting
	\[
	\cE^N_{0001}(x) \defeq \sum_{n=1}^{N} T_n (x),
	\]
	the function $\cE^N_{0001}$ has periodicity $2^{N+3}$ due to \hyperref[periodE]{Corollary~\ref*{periodE}}. 
	When evaluated at integers congruent 1 modulo 4 it thus suffices to know the first $2^{N+1}$ values. 
	We now show that the values $\cE^N_{0001}(4k+1)$ for $k$ from 0 to $2^{N+1}$ are $2^{N-1}$ times 0, followed by $2^{N-1}$ times 3, followed by $2^{N-1}$ times 2 and finally $2^{N-1}$ times 1. 
	
	This obviously holds for the case $N=1$, as then we have 
	$
	T_1= E_3$ which evaluated at $4k+1$ gives the sequence 0, 3, 2, 1 repeated indefinitely. 
	Next, we assume this holds for $\cE^N_{0001}(4k+1)$ and we show that it also holds for $\cE^{N+1}_{0001}(4k+1)$.
	We have by definition
	\[
	\cE^{N+1}_{0001}(x) = T_{N+1} (x)+ \sum_{n=1}^{N} T_n (x).
	\]
	Due to \hyperref[periodE]{Corollary~\ref*{periodE}}, it suffices to look at $T_{N+1}$ evaluated at the values $4k+1$ for $k$ from 0 to $2^{N+2}$, which we found in \eqref{Tn413}. 
	The addition of these values with the assumed values of $\cE^N_{0001}$ now give the desired values of $\cE^{N+1}_{0001}$ modulo 4. This is  illustrated in Table~\ref{tableEN1}.
	
	Hence, for an integer congruent 1 modulo 4, say $4k+1$, we thus find that $\cE^{N}_{0001}(4k+1) \equiv 0 \pmod{4} $ for $N$ such that $4k+1 \leq 2^{N-1}$.
	As $\cE_{0001} =  \lim_{N\to \infty}\cE^{N}_{0001}$, we have $\cE_{0001}(4k+1) \equiv 0 \pmod{ 4}$ for every integer $k$.
	In exactly the same manner one finds that 
	$\cE_{0001}(4k+3) \equiv 1 \pmod{ 4}$.
	
	\begin{table}[!htbp]
		\[
		\begin{array}{c|*{17}c}
		2k+1& 1&3&5 &7&9&11&13&15&17&19&21&23&25&27&29&
		{31}&33\\ \hline 
		E_3 & 0 & 1 & 3& 2 & 2 &3 & 1&0&0&1&3&2&2&3&1&0&0\\
		E_5 & 0&0&1&1&3&3&2&2&2&2&3&3&1&1&0&0&0\\
		E_9& 0&0&0&0&1&1&1&1&3&3&3&3&2&2&2&2&2\\
		E_{17} & 0 &0 &0 &0 &0 &0 &0 & 0 &1&1&1&1&1&1&1&1&3\\
		E_{31} & 0 &0 &0 &0 &0 &0 &0 & 0 &0&0&0&0&0&0&0&0&1
		\end{array}
		\]
		\caption
		{The values (modulo 4) of the polynomials $E_3,E_5\dotsc,E_{31}$ evaluated on the odd integers 1 to 33.}
		\label{tableE2n1}
	\end{table}
	
	\begin{table}[!htbp]
		\[
		\begin{array}{c|*{17}c}
		4k+1& 1&5 &9&13&17&21&25&29 & 33 &37&41&45&49&53&57&
		{61}&65\\ \hline 
		T_1 & 0 & 3 & 2 & 1 & 0 &3 & 2&1&0&3&2&1&0&3&2&1&0\\
		T_2 & 0&1&1&2&2&3&3&0&0&1&1&2&2&3&3&0&0\\
		T_3& 0&0&1&1&1&1&2&2&2&2&3&3&3&3&0&0&0\\
		T_4 & 0&0&0&0&1&1&1&1&1&1&1&1&2&2&2&2&2\\
		T_5 & 0 &0 &0 &0 &0 &0 &0 & 0 &1&1&1&1&1&1&1&1&1\\
		T_6 & 0 &0 &0 &0 &0 &0 &0 & 0 &0&0&0&0&0&0&0&0&1
		\end{array}
		\]
		\caption
		{The values (modulo 4) of the polynomials $T_1,T_2\dotsc,T_6$ evaluated on the integers congruent 1 modulo 4, ranging from 1 to 65.}
		\label{tableTn1}
	\end{table}
	
	\begin{table}[!htbp]
		\[
		\begin{array}{c|*{17}c}
		4k+3& 3&7 &11&15&19&23&27&31 & 35 &39&43&47&51&55&59&
		{63}&67\\ \hline 
		T_1 & 1 & 2 & 3 & 0 & 1 &2 & 3&0&1&2&3&0&1&2&3&0&0\\
		T_2 & 0&3&3&2&2&1&1&0&0&3&3&2&2&1&1&0&0\\
		T_3& 0&0&3&3&3&3&2&2&2&2&1&1&1&1&0&0&0\\
		T_4 & 0&0&0&0&3&3&3&3&3&3&3&3&2&2&2&2&2\\
		T_5 & 0 &0 &0 &0 &0 &0 &0 & 0 &3&3&3&3&3&3&3&3&3\\
		T_6 & 0 &0 &0 &0 &0 &0 &0 & 0 &0&0&0&0&0&0&0&0&3
		\end{array}
		\]
		\caption
		{The values (modulo 4) of the polynomials $T_1,T_2\dotsc,T_6$ evaluated on the integers congruent 3 modulo 4, ranging from 3 to 67.}
		\label{tableTn3}
	\end{table}
	\begin{table}[!htbp]
		\[
		\begin{array}{c*{16}c}
		&\overbrace{\rule{1.0cm}{0pt}}^{2^{N-1}} &\overbrace{\rule{1.0cm}{0pt}}^{2^{N-1}} & \overbrace{\rule{1.0cm}{0pt}}^{2^{N-1}} & \overbrace{\rule{1.0cm}{0pt}}^{2^{N-1}} & \overbrace{\rule{1.0cm}{0pt}}^{2^{N-1}} & \overbrace{\rule{1.0cm}{0pt}}^{2^{N-1}} &  \overbrace{\rule{1.0cm}{0pt}}^{2^{N-1}} & \overbrace{\rule{1.0cm}{0pt}}^{2^{N-1}} \\[-3pt]
		\cE^{N}_{0001}(4k+1) & 0  \dots  0 & 3  \dots  3 & 2  \dots  2& 1  \dots  1& 0  \dots  0 & 3  \dots  3 & 2  \dots  2& 1  \dots  1 \\
		T_{N+1}(4k+1) & 0  \dots  0 & 1  \dots  1 & 1  \dots  1& 2  \dots  2& 2  \dots  2  & 3  \dots  3 & 3  \dots  3& 0  \dots  0\\ \hline
		\cE^{N+1}_{0001}(4k+1) & 0  \dots  0 & 0  \dots  0 & 3  \dots  3& 3  \dots  3& 2  \dots  2& 2  \dots  2& 1  \dots  1& 1  \dots  1\\[-9pt]
		& \multicolumn{2}{c}{\underbrace{\rule{2.4cm}{0pt}}_{2^{N}}} & \multicolumn{2}{c}{\underbrace{\rule{2.4cm}{0pt}}_{2^{N}}} & \multicolumn{2}{c}{\underbrace{\rule{2.4cm}{0pt}}_{2^{N}}} & \multicolumn{2}{c}{\underbrace{\rule{2.4cm}{0pt}}_{2^{N}}}
		\end{array}
		\]
		\caption
		{The values (modulo 4) of $\cE^{N+1}_{0001}$ evaluated on the first $2^{N+2}$ integers congruent 1 modulo 4.}
		\label{tableEN1}
	\end{table}
\end{proof}

The proof of {\hyperref[moddc]{Theorem~\ref*{moddc}}} is analogous to the one of \hyperref[moddb]{Theorem~\ref*{moddb}}, now using the following periodicity property of the polynomials $D_n$.

\begin{corollary}\label{periodD}
	For natural numbers $n,N\in \mZ_{\geq0}$ such that $n \leq 2^N$, and $x\in \mZ_{\geq0}$, one has
	\[
	D_{n}\big(x\!+\!\tfrac12\big) \equiv D_n\big(x\!+\!\tfrac12+2^{N+2}\big) \pmod{ 4},
	\]
	and
	\[
	2\, D_{n}\big(x\!+\!\tfrac12\big) \equiv 2\, D_n\big(x\!+\!\tfrac12+2^{N+1}\big) \pmod{ 4},
	\]
\end{corollary}
\begin{proof}
	Writing \eqref{Dnx} as
	\begin{equation*}
	D_n\big(x\!+\!\tfrac12\big) =  \frac{1}{(2n)!} \prod_{k=0}^{n-1} \big((x\!+\!\tfrac12)^2-(k\!+\!\tfrac12)^2\big)=  \frac{1}{(2n)!} \prod_{k=0}^{2n-1} (x-n+1+k),
	\end{equation*}
	the periodicity follows by the same reasoning as used in the proof of \hyperref[period]{Proposition~\ref*{period}}.
\end{proof}

\subsection{Clifford analysis}
\label{secA3}

To show \hyperref[intrC]{Proposition~\ref*{intrC}} we first give some auxiliary results.
For these results, let $T$ be an integral transform 
\[
T \lbrack f(x) \rbrack (y) = \frac{1}{(2\pi)^{m/2}} \int_{\mathbb{R}^m}\! K(x,y) \, f(x) \, \mathrm{ d}x
\]
with kernel given by
\begin{equation*}
K(x,y) = A(w,z) +  (\ux \wedge \uy) \  B(w,z)
\end{equation*}
where, for $\alpha_{k}, \beta_{k} \in \mC$,
\begin{align*}
A(w,z) &=  {2}^{\lambda}\, \Gamma ( \lambda )    \sum_{k=0}^{+\infty} (k+\lambda)\, \alpha_{k} \ z^{-\lambda}J_{k+\lambda}( z) C^{\lambda}_{k}(w)\\
B(w,z) &=  2^{\lambda+1}\, \Gamma ( \lambda +1 )    \sum_{k=1}^{+\infty}  (k+\lambda )\, \beta_{k} \ z^{-\lambda-1}J_{k+\lambda}( z) C^{\lambda+1}_{k-1}(w).
\end{align*}
The eigenvalues of $T$ can be determined using the following proposition, which is a generalization of Bochner's formulas for the classical Fourier transform.

\begin{proposition}
	\label{Bochner}
	Let $M_{k} \in \cM_{k}$ be a spherical monogenic of degree $k$. Let $f(x)= f_0(|x|)$ be a real-valued radial function in 
	$\cS(\mR^m)$. Further, put $\underline{\xi}= \ux/|x|$, $\underline{\eta} = \uy/|y|$ and $r = |x|$. Then one has
	\[
	T\lbrack f(x)M_{k}(x) \rbrack (y) =   
	\left(\alpha_{k} - k\, \beta_k \right)  M_{k}({\eta}) \int_{0}^{+\infty} r^{m+k-1}f_0(r)   z^{-\lambda} J_{k + \lambda}(z)  \, \mathrm{d}r
	\]
	and
	\[
	T \lbrack f(x) \ux M_{k}(x)\rbrack (y) =
	(  \alpha_{k+1} + (k+m-1) \beta_{k+1} )  \underline{\eta}  M_{k}({\eta})\int_{0}^{+\infty} r^{m+k}f_0(r)   z^{-\lambda} J_{k +1+ \lambda}(z)  \, \mathrm{d}r
	\]
	with $z= r |y|$ and $\lambda= (m-2)/2$.
\end{proposition}
\begin{proof}
	The proof relies on \eqref{reprod} and goes along similar lines as the proof of Theorem 6.4 in \cite{DBXu}.
\end{proof}
We then have:
\begin{theorem}
	\label{eigenvalues}
	The functions $\{  \psi_{j,k,\ell}  \}$ are eigenfunctions of $T$.
	One has, putting $\beta_{0}=0$,
	\begin{align}\label{eigenvalue eq}
	\begin{split}
	(T \psi_{2p,k,\ell}) (y) & = 
	(-1)^{p}\left( \alpha_{k} - {k} \, \beta_{k} \right)\psi_{2p,k,\ell}(y)\\
	(T \psi_{2p+1,k,\ell}) (y) & = 
	(-1)^{p}\left(  \alpha_{k+1} + ({k+m-1}) \beta_{k+1}\right)\psi_{2p+1,k,\ell}(y).
	\end{split}
	\end{align}
\end{theorem}
\begin{proof}
	The functions $\{   \psi_{j,k,\ell}   \}$ are of the form $f(x) M_{k}(x)$ or $f(x) \ux M_{k}(x)$ with $f(x)= f_0(|x|)$ a real-valued radial function in $\cS(\mR^m)$.
	Hence we can apply \hyperref[Bochner]{Proposition~\ref*{Bochner}} and we find, using $\lambda= (m-2)/2$, 
	\[
	(T \psi_{2p,k,\ell}) (y) = \left(  \alpha_{k} - {k}\, \beta_k \right)  M_{k}^{(\ell)}({\eta})
	\int_{0}^{\infty} r^{2\lambda+1+k}  L_{p}^{k+\lambda}( r^{2}) \,  e^{- r^{2}/2}   z^{-\lambda} J_{k + \lambda}(z)  \, \mathrm{d}r.
	\]
	Substituting $z= r |y|$, the integral becomes
	\[
	\int_{0}^{\infty}r^{2\lambda+1+k}  L_{p}^{k+\lambda}(r^{2}) \  e^{- r^{2}/2}   (r |y|)^{-\lambda} J_{k + \lambda}(r |y|)  \, \mathrm{d} r.
	\]
	Now we can apply the identity \eqref{lagbes} 
	to give the final result of 
	\begin{align*}
	(T \psi_{2p,k,\ell}) (y) &=  (-1)^{p}\left( \alpha_{k} - {k} \, \beta_{k} \right)  M_{k}^{(\ell)}(y)  L_{p}^{k+\lambda}( |y|^{2}) e^{-|y|^{2}/2} \\
	&=  {(-1)^{p}}  \left( \alpha_{k} - {k} \, \beta_{k} \right) \psi_{2p,k,\ell}(y).
	\end{align*}
	The expression for $(T\psi_{2p+1,k,\ell}) (y)$ follows similarly.
\end{proof}

Using these results we finally arrive at:

\begin{proof}[Proof of {\hyperref[intrC]{Proposition~\ref*{intrC}}}]
	Denote the eigenvalues of $T$ by $  \mu_{j,k}$ for ${j,k\in\mZ_{\geq0}}$. 
	From \hyperref[eigenvC]{Lemma~\ref*{eigenvC}}, we see that these eigenvalues satisfy
	\[
	\mu_{j+2,k} = -\mu_{j,k}.
	\]
	Moreover, as $\mu_{j+1,k} = i\, \mu_{j,k}$ we find that putting 
	\begin{equation*}
	\alpha_{k}   =\frac{1}{2(k+\lambda)} \big(i \, k \, \mu_{0,k-1}+(k+2\lambda) \mu_{0,k}\big),\qquad
	\beta_{k}
	= 
	\frac{ 1}{2(k+\lambda)}\big(i \, \mu_{0,k-1}- \mu_{0,k}\big)
	\end{equation*}
	in \eqref{eigenvalue eq}, gives an integral transform that coincides with $T$ on the eigenfunction basis.
\end{proof}

We conclude with a note relevant to the space $L^{2}(\mathbb{R}^m) \otimes \cC l_{m}$ in \hyperref[theoUPC]{Theorem~\ref*{theoUPC}}. 
This space is equipped with the inner product 
\begin{equation}\label{IPC}
\langle f, g \rangle = \left[ \int_{\mR^{m}} \overline{f^{c}} \, g \, \mathrm{d}x \right]_{0}.
\end{equation}
Here, $u \mapsto \bar{ u }$ is the anti-involution on the Clifford algebra $\cC l_{m}$ defined by
\begin{equation*}
\overline{uv} = \overline{v} \, \overline{u} 
\qquad  \mbox{and} \qquad
\overline{e_{j}} = -e_{j} \quad (j= 1,\ldots, m).
\end{equation*}
Furthermore, 
$f^{c}$ denotes the complex conjugate of the function $f$ and $u \mapsto [u]_{0}$ is the projection on the space of $0$-vectors (scalars). The functions $\{\psi_{j,k,\ell}\}$ defined in formula (\ref{basis}) are after suitable normalization an orthonormal basis for $L^{2}(\mathbb{R}^m) \otimes \cC l_{m}$ (see e.g.~\cite{MR926831}), satisfying
\[
\langle \psi_{j_{1},k_{1}, \ell_{1}} , \psi_{j_{2},k_{2}, \ell_{2}} \rangle = \delta_{j_{1} j_{2}} \delta_{k_{1} k_{2}}\delta_{\ell_{1} \ell_{2}}.
\]

\end{document}